\documentclass{article}
\usepackage{amsmath, amssymb, latexsym, amscd, amsthm,amsfonts,amstext}
\usepackage[mathscr]{eucal}
\usepackage{graphics, graphpap}
\usepackage{makeidx}
\usepackage{xspace}
\usepackage{array, tabularx, longtable}
\usepackage{multicol,color}
\usepackage{color}

\def\HH{{\mathbb H}}

\def\RR{{\mathbb R}}

\def\suml{\sum\limits}
\def\intl{\int\limits}
\newtheorem{theorem}{\bf Theorem}[section]
\newtheorem{define}[theorem]{\bf Definition}
\newtheorem{corollary}[theorem]{\bf Corollary}
\newtheorem{lemma}[theorem]{\bf Lemma}
\newtheorem{proposition}[theorem]{\bf Proposition}
\newtheorem{example}[theorem]{\bf Example}

\usepackage{amsmath, amssymb, latexsym, amscd, amsthm,amsfonts,amstext}
\usepackage[mathscr]{eucal}
\usepackage{graphics, graphpap}
\usepackage{makeidx}
\usepackage{xspace}
\usepackage{array, tabularx, longtable}
\usepackage{multicol,color}
\def\Om{\Omega}
\def\om{\omega}

\def\eps{\varepsilon}

\def\intl{\int\limits}
\def\suml{\sum\limits}

\def\({\left(}
\def\){\right)}

\def\suml{\sum\limits}

\def\intl{\int\limits}
\def\prodl{\prod\limits}
\def\grad{\triangledown}
 
\begin{document}
\title{Local and global sharp gradient estimates for weighted
$p$-harmonic functions}
\author{Nguyen Thac Dung and Nguyen Duy Dat}
\date{\today} 
\maketitle
\begin{abstract}
Let $(M^n, g, e^{-f}dv)$ be a smooth metric measure space of dimensional $n$. Suppose that $v$ is a positive weighted $p$-eigenfunction associated to the eigenvalue $\lambda_{1,p}$ on $M$, namely
$$ e^{f}div(e^{-f}|\nabla v|^{p-2}\nabla v)=-\lambda_{1,p}v^{p-1}.$$
in the distribution sense. We first give a local gradient estimate for $v$ provided the $m$-dimmensional Bakry-\'{E}mery curvature $Ric_f^{m}$ bounded from below. Consequently, we show that when $Ric_f^m\geq0$ then $v$ is constant if $v$ is of sublinear growth. At the same time, we prove a Harnack inequality for weighted $p$-harmonic functions. Moreover, we show global sharp gradient estimates for weighted $p$-eigenfunctions. Then we use these estimates to study geometric structures at infinity when the first eigenvalue $\lambda_{1,p}$ obtains its maximal value. Our achievements generalize several results proved ealier by Li-Wang, Munteanu-Wang,...(\cite{LW1, LW2, MW1, MW2}).\\

\vspace{0.2cm}

\end{abstract}
\vskip0.4cm
\section{Introduction}
The local Cheng-Yau gradient estimate is a standard result in Riemannian geometry, see \cite{CY75}, also see \cite{SY94}. It asserts that if $M$ be an $n$ dimensional complete Riemannian manifold with $Ric\geq -(n-1)\kappa$ for some $\kappa\geq0$, for $u: B(o, R)\subset M\to\RR$ harmonic and positive then there is a constant $c_n$ depending only on $n$ such that 
\begin{equation}\label{e11}
\sup\limits_{B(o,R/2)}\frac{|\nabla u|}{u}\leq c_n\frac{1+\ \sqrt[]{\kappa}R}{R}.
\end{equation}
Here $B(o,R)$ stands for the geodesic centered at a fixed point $o\in M$. Notice that when $\kappa=0$, this implies that a harmonic function with sublinear growth on a manifold with non-negative Ricci curvature is constant. This result is clearly sharp since on $\RR^n$ there exist harmonic functions which are linear.

Cheng-Yau's method is then extended and generalized by many mathematicians. For example, Li-Yau (see \cite{LY}) obtained a gradient estimate for heat equations. Cheng (see \cite{Cheng}) and H. I. Choi (see \cite{Choi}) proved gradient estimates for harmonic mappings, etc. We refer the reader to survey \cite{S7} for an overview of the subject. 

When $(M^n, g, e^{-f}dv)$ are smooth metric measure spaces, it is very natural to find similar results. Recall that the triple $(M^n, g, e^{-f}d\mu)$ is called a smooth metric measure space if $(M, g)$ is a Riemannian manifold, $f$ is a smooth function on $M$ and $d\mu$ is the volume element induced by the metric $g$. On $M$, we consider the differential operator $\Delta_f$, which is called $f-$Laplacian and given by
$$ \Delta_f\ \cdot:=\Delta\cdot-\left\langle \grad f,\ \grad\cdot\right\rangle.  $$
It is symmetric with repect to the measure $e^{-f}d\mu$. That is,
$$ \intl_M\left\langle \grad \varphi, \grad\psi\right\rangle e^{-f}=-\intl_M(\Delta_f\varphi)\psi e^{-f}, $$
for any $\varphi, \psi\in C^\infty_0(M)$. Smooth metric measure spaces are also called manifolds with density. By $m$-dimensional Bakry-\'{E}mery Ricci tensor we mean
$$ Ric_f^m=Ric+Hessf-\frac{\nabla f\otimes\nabla f}{m-n}, $$
for $m\geq n$. Here $m=n$ iff $f$ is constant. The $\infty-$Barky-\'{E}mery tensor is refered as 
$$ Ric_f=Ric+Hessf. $$  

Brighton (see \cite{Bri}) gave a gradient estimate of positive weighted harmonic function, as a consequence, he proved that any bounded weighted harmonic function on a smooth metric measure space with $Ric_f\geq0$ has to be constant. Later, Munteanu and Wang refined Brighton's argument and proved that positive $f$-harmonic function of sub-exponential growth on smooth metric measure space with nonnegative $Ric_f$ must be a constant function. Moreover, Munteanu and Wang also applied the De Giorgi-Nash-Moser theory to get a sharp gradient estimate for any positive $f$-harmonic function provided that the weighted function $f$ is at most linear growth (see \cite{MW1, MW2} for further results). On the other hand, Wu derived a Li-Yau type estimate for parabolic equations. He also made some results for heat kernel (see \cite{Wu1, Wu14} for the details.).   

From a variational point of view, $p$-harmonic function, or more general weighted $p$-harmonic functions are natural extensions of harmonic functions, or weight harmonic functions, respectively. Compared with the theory for (weighted) harmonic functions, the study of (weighted) $p$-harmonic functions is generally harder, even though elliptic, is degenerate and the regularity results are far weaker. We refer the reader to \cite{Mo, KL} for the connection between $p$-harmonic functions and the inverse mean curvature flow. For the weighted $p$-harmonic function, Wang (see \cite{W}) estimated eigenvalues of this operator. On the other hand, Wang, Yang and Chen (see \cite{WYC}) shown gradient estimates and entropy formulae for weighted $p$-heat equations. Their works generalized Li's and Kotschwar-Ni's results (see \cite{LXD, KL}).

In this paper, motived by Wang-Zhang's gradient estimate for the $p$-harmonic function, we give the following result on local gradient estimates of weighted $p$-eigenfunctions.
\begin{theorem}\label{maintheorem}
Let $(M^{n}, g, e^{-f})$ be a smooth metric measure space of dimension $n$ with $Ric_f^{m}\geq-(m-1)\kappa$. Suppose that $v$ is a positive smooth weighted $p$-eigenfunction with repsect to the eigenvalue $\lambda_{1,p}$ on the ball $B_R=B(o, R)\subset M$. This means
$$ e^{f}div(e^{-f}|\nabla v|^{p-2}\nabla v) =-\lambda_{1,p}v^{p-1}$$
on $B_R=B(o, R)\subset M$. Then there exists a constant $C=C(p, m, n)$ such that 
\begin{equation}\label{maine}
 \frac{|\nabla v|}{v}\leq \frac{C(1+\ \sqrt[]{\kappa}R)}{R} \quad \text{ on }\quad B(o, R/2).
\end{equation}
\end{theorem}
Moreover, using these local gradient estimates, we can obtain sharp gradient estimates for weighted $p$-eigenfunctions as follows.
\begin{theorem}\label{2ndtheorem}
Let $(M^n, g, e^{-f}dv)$ be an $n$-dimensional complete noncompact manifold with $Ric_f^{m}\geq-(m-1)$. 
 If $v$ is a positive weighted $p$-eigenfunction with respect to the first eigenvalue $\lambda_{1,p}$, then 
$$ |\nabla\ln v|\leq y. $$
Here $y$ is the unique positive root of the equation
$$ (p-1)y^{p}-(m-1)y^{p-1}+\lambda_{1,p}=0. $$  
\end{theorem}
A directly consequence of the theorem \ref{2ndtheorem} is a sharp gradient estimate for positive weighted $p$-harmonic function.
\begin{corollary}
Let $(M^n, g, e^{-f}dv)$ be an $n$-dimensional complete noncompact manifold with $Ric_f^{m}\geq-(m-1)$. 
If $v$ is a positive weighted $p$-harmonic function
then 
$$ |\nabla\ln v|\leq \frac{m-1}{p-1}. $$
\end{corollary}
 The sharpness of the estimate is demonstrated by the below example.
\begin{example}
Let $M^n=\RR\times N^{n-1}$ with a warped product metric 
$$ ds^2=dt^2+e^{2t}ds_N^2, $$
where $N$ is a complete manifold with non-negative Ricci curvature. Then it can be directly checked that $Ric_M\geq-(n-1)$  (See \cite{Liwang} for details of computation). Moreover, we have
$$ \Delta =\frac{\partial^2}{\partial t^2}+(n-1)\frac{\partial}{\partial t}+e^{-2t}\Delta_N. $$
Choose weighted function $f=-(m-n)t$, then the $m$-dimensional Bakry-\'{E}mery curvature is bounded from below by $-(m-1)$.

 Let $v(t, x)=e^{-at}$, where $\frac{m-1}{p}\leq a\leq\frac{m-1}{p-1}$, we have 
$$ \frac{|\nabla v|}{v}=a. $$
It is easy to show that
$$\begin{aligned} 
|\nabla v|^{p-2}\left\langle\nabla v, \nabla f\right\rangle&=(m-n)a^{p-1}v^{p-1}\\
\left\langle \nabla|\nabla v|^{p-2}, \nabla v\right\rangle&=(p-2)a^pv^{p-1}\\
|\nabla v|^{p-2}\Delta v&=(1-n+a)a^{p-1}v^{p-1}. 
\end{aligned} $$
Hence, 
$$ e^fdiv(e^{-f}|\nabla v|^{p-2}\nabla v)=((p-1)a-(m-1))a^{p-1}v^{p-1}. $$
This implies that 
$$ \lambda_{1,p}=(m-1-(p-1)a)a^{p-1}, $$
or equivalently,
$$ (p-1)a^p-(m-1)a^{p-1}+\lambda_{1,p}=0. $$
\end{example}
It is also very interesting to ask what are geometric structures of manifolds with $\lambda_{1,p}$ achieving its maximal value. When $f$ is constant, this problem has been studied by Li-Wang, Sung-Wang in \cite{LW1, LW2, SW}. In this paper, we prove a generalization of their results. 
\begin{theorem}\label{3rdtheorem}Let $(M^n, g, e^{-f}dv)$ be a smooth metric measure space of dimension $n\geq2$. Suppose that $Ric_f^m\geq-(m-1)$ and $\lambda_{1,p}=\left(\frac{m-1}{p}\right)^p$. Then either $M$ has no $p$-parabolic ends or $M=\RR\times N^{n-1}$ for some compact manifold $N$. Here the definition of $p$-parabolic ends is given in the section 3.
\end{theorem} 
On the other hand, if we only assume that the Bakry-\'{E}mery curvature is bounded from below, we also can give a upper bound estimate of the first eigenvalue. For example, we consider smooth metric measure spaces and prove that if $Ric_f\geq0$ then $\lambda_{1, p}\leq\left(\frac{a}{p}\right)^{p}$, where $a$ is the linear growth rate of $f$. Moreover, we also show that this estimate is optimal. When $\lambda_{1,p}$ is maximal we obtain the following theorem.
\begin{theorem}\label{rigidity1}
Let $(M, g, e^{-f}dv)$ be a smooth metric measure space with $Ric_f\geq0$. Suppose that $\lambda_{1,p}=\left(\dfrac{a}{p}\right)^{p}$, where $a$ is the linear growth rate of $f$. Then, either $M$ is connected at infinity or $M=\RR\times N$ where $N$ is a compact manifold.
\end{theorem}
This theorem is a generalization of Munteanu and Wang's theory on weighted harmonic functions on smooth metric measure spaces (see \cite{MW1}). 

This paper is organized as follows. In the section 2, we give a proof of the main theorem \ref{maintheorem} by using the Moser's iteration. As its applications, we show a Liouville property and a Harnack inequality for weighted $p$-harmonic functions. In the section 3, we prove the theorem \ref{2ndtheorem}. The proof the theorem \ref{3rdtheorem} is given in the section 4. In the section 5, we investigate smooth metric measure spaces with Bakry-\'{E}mery curvatures bounded from below. We generalize Munteanu and Wang's results in \cite{MW1, MW2} by using the theory of weighted $p$-harmonic functions. 
\section{Local gradient estimates for weighted $p$-eigenfunctions on $(M, g, e^{-f}d\mu)$}
\setcounter{equation}{0}
Suppose that $(M, g, e^{-f}dv)$ is a smooth metric measure space, and $\Omega\subset M$ be an open subset. Let $v$ be a positive weighted $p$-eigenfunction function with respect to the first eigenvalue $\lambda_{1,p}$, for $p>1$, namely, $v\in W^{1,p}_{loc}(\Omega)$ satisfying the following weighted $p$-Laplacian equation,
\begin{equation}\label{flap}
 \Delta_{p,f}v:=e^{f}div(e^{-f}|\nabla v|^{p-2}\nabla v)=-\lambda_{1,p}v^{p-1}. 
\end{equation}
in the distribution sense, i. .e,
$$ \int_\Omega \left\langle|\nabla v|^{p-2}\nabla v, \nabla\varphi\right\rangle e^{-f}dv=\lambda_{1,p}\int_\Omega v^{p-1}\varphi e^{-f}dv  $$
where $\varphi\in W^{1,p}_{0}(\Omega)$. The regularity of solutions of the equation \eqref{flap} implies $v\in{\mathcal C}^{1, \alpha}$ for some $\alpha>0$, for example, see \cite{Tol}. Morevoer, it is well-known that $v\in W^{2,2}_{loc}$ if $p\geq 2$ and $v\in W^{2, p}_{loc}$ if $1<p<2$. In addition, $v$ is smoth outside the set $\left\{\nabla v=0\right\}$.
 
 Note that, when $\lambda_{1,p}=0$ then $v$ is called a weighted $p$-harmonic function. Let $u=-(p-1)\log v$, then $v=e^{-u/(p-1)}$.  It is easy to see that $u$ satisfies 
$$\begin{aligned}
e^{f}div(e^{-f}|\nabla u|^{p-2}\nabla u)=|\nabla u|^{p}+\lambda_{1,p}(p-1)^{p-1}.
\end{aligned} $$  
Put $h:=|\grad u|^2$, the above equation can be rewritten as follows.
\begin{equation}\label{dd1}
\left(\frac{p}{2}-1\right)h^{p/2-2}\left\langle\nabla h, \nabla u\right\rangle+h^{p/2-1}\Delta_fu =h^{p/2}+\lambda_{1,p}(p-1)^{p-1}. 
\end{equation}
Assume that $h>0$. As in \cite{KL}, \cite{Mo}, we consider the below operator
$$ {\cal L}_f(\psi):=e^fdiv\Big(e^{-f}h^{p/2-1}A(\nabla \psi)\Big)-ph^{p/2-1}\left\langle \nabla u, \nabla\psi\right\rangle,  $$
where 
$$ A=id+(p-2)\frac{\nabla u\otimes\nabla u}{|\nabla u|^{2}}. $$

We have the following lemma and the proof is by direct computation.
\begin{lemma}\label{d2}
\begin{equation}\label{e1}
{\cal L}_f(h)=2h^{p/2-1}(u_{ij}^2+\text{Ric}_f(\nabla u, \nabla u))+\left(\frac{p}{2}-1\right)h^{p/2-2}|\nabla h|^{2}.
\end{equation}
\end{lemma}
\begin{proof}
By the definition of ${\cal L}_f$, we have
$$ \begin{aligned}
{\cal L}_f(h)
=&e^fdiv\left(e^{-f}h^{\frac{p}{2}-1}\left(\nabla h+(p-2)\frac{\left\langle\nabla u, \nabla h\right\rangle}{h}\nabla u \right)\right)-ph^{\frac{p}{2}-1}\left\langle\nabla u, \nabla h\right\rangle\\
 =&h^{\frac{p}{2}-1}\Delta h+(p-2)h^{\frac{p}{2}-1}div\left(h^{-1}\left\langle\nabla u, \nabla h\right\rangle\nabla u \right)\\
&\quad +e^f\left\langle\nabla\left(e^{-f}h^{\frac{p}{2}-1}\right), \nabla h+(p-2)h^{-1}\left\langle\nabla u, \nabla h\right\rangle\nabla u \right\rangle -ph^{\frac{p}{2}-1}\left\langle\nabla u, \nabla h\right\rangle\\
=&h^{\frac{p}{2}-1}\Delta h+(p-2)h^{\frac{p}{2}-2}\left\langle\nabla u, \nabla h\right\rangle\Delta u-(p-2)h^{\frac{p}{2}-3}\left\langle\nabla u, \nabla h\right\rangle^2\\
&\ +(p-2)h^{\frac{p}{2}-2}(u_{ij}h_iu_j+h_{ij}u_iu_j)+\left(\frac{p}{2}-1\right)h^{\frac{p}{2}-2}|\nabla h|^2+(p-2)\left(\frac{p}{2}-1\right)h^{\frac{p}{2}-3}\left\langle\nabla u, \nabla h\right\rangle^2\\
&\quad -h^{\frac{p}{2}-1}\left\langle\nabla f, \nabla h\right\rangle-(p-2)h^{\frac{p}{2}-2}\left\langle\nabla u, \nabla h\right\rangle\left\langle\nabla f, \nabla u\right\rangle-  ph^{\frac{p}{2}-1}\left\langle\nabla u, \nabla h\right\rangle.\\   
\end{aligned} $$
Hence,
$$ \begin{aligned}
{\cal L}_f(h)
=&h^{\frac{p}{2}-1}\Delta_fh+\left(\frac{p}{2}-1\right)h^{\frac{p}{2}-2}|\nabla h|^2\\
&\ +(p-2)h^{\frac{p}{2}-2}\left\langle\nabla h, \nabla u\right\rangle\Delta_fu+(p-2)\left(\frac{p}{2}-2\right)h^{\frac{p}{2}-3}\left\langle\nabla h, \nabla u\right\rangle^2\\
&\quad +(p-2)h^{\frac{p}{2}-2}(u_{ij}h_iu_j+h_{ij}u_iu_j)- ph^{\frac{p}{2}-1}\left\langle\nabla u, \nabla h\right\rangle\\
=&2h^{\frac{p}{2}-1}\left(u_{ij}^2+Ric_f(\nabla u, \nabla u)\right)+\left(\frac{p}{2}-1\right)h^{\frac{p}{2}-2}|\nabla h|^2\\
&\ +(p-2)h^{\frac{p}{2}-2}\left\langle\nabla h, \nabla u\right\rangle\Delta_fu +2h^{\frac{p}{2}-1}\left\langle\nabla\Delta_fu, \nabla u\right\rangle +(p-2)\left(\frac{p}{2}-2\right)h^{\frac{p}{2}-3}\left\langle\nabla h, \nabla u\right\rangle^2\\
&\quad +(p-2)h^{\frac{p}{2}-2}(u_{ij}h_iu_j+h_{ij}u_iu_j)- ph^{\frac{p}{2}-1}\left\langle\nabla u, \nabla h\right\rangle.
\end{aligned} $$
Here we used the Bochner identity
$$\Delta_fh=\Delta_f|\nabla u|^2=2\left(u_{ij}^2+Ric_f(\nabla u, \nabla u)\right)+2\left\langle\nabla\Delta_fu, \nabla u\right\rangle  $$
in the last equation.

On the other hand, differentiating both side of \eqref{dd1} then multiplying the obtained results by $\nabla u$, we have
$$ \begin{aligned}
\frac{p}{2}h^{\frac{p}{2}-1}\left\langle\nabla h, \nabla u\right\rangle= &\left(\frac{p}{2}-1\right)h^{\frac{p}{2}-2}\left\langle\nabla h, \nabla u\right\rangle\Delta_fu+h^{\frac{p}{2}-1}\left\langle\nabla\Delta_fu, \nabla u\right\rangle\\
  &+\left(\frac{p}{2}-2\right)\left(\frac{p}{2}-1\right)h^{\frac{p}{2}-3}\left\langle\nabla h, \nabla u\right\rangle^2+\left(\frac{p}{2}-1\right)h^{\frac{p}{2}-2}(u_{ij}h_iu_j+h_{iju_iu_j}) 
\end{aligned} $$
Combining this equation and the above equation, we are done.
\end{proof}
Now, suppose that $v$ is a weighted $p$-eigenfunction with repect to the first eigenvalue $\lambda_{1, p}=0$. We choose a local orthonormal frame $\{e_i\}$ with $e_1=\nabla u/|\nabla u|$ then $$2hu_{11}=\left\langle\nabla u, \nabla h\right\rangle, \quad
 \sum\limits_{i=1}^{n}u_{1i}^{2}=\frac{1}{4}\frac{|\nabla h|^{2}}{h}. $$
Then \eqref{dd1} can be read as 
$$ (p-1)u_{11}+\suml_{i=2}^nu_{ii}=h+\left\langle\nabla f, \nabla u\right\rangle +|\nabla u|^{2-p}\lambda_{1,p}(p-1)^{p-1}. $$
Therefore
$$ \begin{aligned}
u_{ij}^2
&\geq \suml_{i=1}^nu_{1i}^2+\suml_{i=2}^nu_{ii}^2\\
&\geq \suml_{i=1}^nu_{1i}^2+\frac{1}{n-1}\left(\suml_{i=2}^nu_{ii}\right)^2\\
&=\suml_{i=1}^nu_{1i}^2+\frac{1}{n-1}\left(h+\lambda_{1,p}(p-1)^{p-1}|\nabla u|^{2-p}-(p-1)u_{11}+\left\langle \nabla f, \nabla u\right\rangle \right)^2\\
&\geq \suml_{i=1}^nu_{1i}^2+\frac{1}{n-1}\left(\frac{(h+\lambda_{1,p}(p-1)^{p-1}|\nabla u|^{2-p}-(p-1)u_{11})^{2}}{1+\frac{m-n}{n-1}}-\frac{\left\langle\nabla u, \nabla f\right\rangle^{2} }{\frac{m-n}{n-1}}\right)\\
&\geq \suml_{i=1}^nu_{1i}^2+\frac{1}{n-1}\left(\frac{h^{2}+2h(\lambda_{1,p}(p-1)^{p-1}|\nabla u|^{2-p}-(p-1)u_{11})}{1+\frac{m-n}{n-1}}-\frac{\left\langle\nabla u, \nabla f\right\rangle^{2} }{\frac{m-n}{n-1}}\right)\\
&\geq \frac{1}{m-1}h^2-\frac{2(p-1)}{m-1}hu_{11}+\suml_{i=2}^nu_{1i}^2-\frac{\left\langle\nabla f, \nabla u\right\rangle^{2} }{m-n}.
\end{aligned} $$
Note that we used $(a-b)^{2}\geq\frac{a^2}{1+\delta}-\frac{b^{2}}{\delta}$ for $\delta >0$ in the fourth inequality. Again, by using the identities
$$ 2hu_{11}=\left\langle \nabla u, \nabla h\right\rangle, \quad \suml_{i=1}^nu_{1i}^2=\frac{1}{4}\frac{|\nabla h|^2}{h}  $$
we conclude that
$$u_{ij}^2\geq  \frac{1}{m-1}h^2-\frac{p-1}{m-1}\left\langle\nabla u,\nabla h\right\rangle +\frac{1}{4}\frac{|\nabla h|^2}{h}-\frac{\left\langle\nabla f, \nabla u\right\rangle^{2} }{m-n}.$$
Assume that $\text{Ric}_f^{m}\geq-(m-1)\kappa$, we infer
$$ \begin{aligned}
{\cal L}_f(h)
&=2h^{p/2-1}(u_{ij}^2+\text{Ric}_f(\nabla u, \nabla u))+\left(\frac{p}{2}-1\right)h^{p/2-2}|\nabla h|^{2}\\
&\geq 2h^{p/2-1}\left(\frac{1}{m-1}h^2-\frac{p-1}{m-1}\left\langle\nabla u,\nabla h\right\rangle +\frac{1}{4}\frac{|\nabla h|^2}{h}+\left(\text{Ric}_f-\frac{df\otimes df}{m-n}\right)(\nabla u, \nabla u)\right)\\
&\quad\quad\quad+\left(\frac{p}{2}-1\right)h^{p/2-2}|\nabla h|^{2}\\
&\geq-2(m-1)\kappa h^{p/2}+\left(\frac{p+1}{2}-1\right)|\nabla h|^{2}h^{p/2-2}\\
&\quad\quad\quad+\frac{2}{m-1}h^{p/2+1}-\frac{2(p-1)}{m-1}h^{p/2-1}\left\langle\nabla u, \nabla h\right\rangle\\
&\geq-2(m-1)\kappa h^{p/2}
+\frac{2}{m-1}h^{p/2+1}-\frac{2(p-1)}{m-1}h^{p/2-1}\left\langle\nabla u, \nabla h\right\rangle\\
 \end{aligned} $$
The above equation holds wherever $h$ is strictly positive. Let $K=\{x\in M, h(x)=0\}$. Then for any non-negative function $\psi$ with compact support in $\Om\setminus K$, we have
\begin{align}
&\int_\Omega \Big< h^{p/2-1}\nabla h+(p-2)h^{p/2-2}\left\langle\nabla u, \nabla h\right\rangle\nabla u, \nabla\psi \Big> e^{-f}\notag\\
&\quad\quad+p\int_\Omega h^{p/2-1}\left\langle\nabla u, \nabla h\right\rangle\psi e^{-f}
+\frac{2}{m-1}\int_\Omega h^{p/2+1}\psi e^{-f}\notag\\
\leq\ 
& 2(m-1)\kappa\int_\Omega h^{p/2}\psi e^{-f}+\frac{2(p-1)}{m-1}\int_\Omega h^{p/2-1}\left\langle \nabla u, \nabla h\right\rangle\psi e^{-f}  \label{e23}
\end{align} 
In order to consider the cases $h=0$, for $\eps>0, b>2$ we choose $\psi=h^b_\eps\eta^2$ where $h_\eps=(h-\eps)^+, \eta\in C^\infty_0(B_R)$ is non-negative, $0\leq\eta\leq1, b$ is to be determined later. Then direct computation shows that
$$ \nabla\psi=bh^{b-1}_\eps\nabla h\eta^2+2h_\eps^b\eta\nabla\eta. $$
Plugging this identity into \eqref{e23}, we have
\begin{align}
{~}&b\left(\int_{\Omega}\Big(h^{\frac{p}{2}-1}h_\eps^{b-1}|\nabla h|^{2}+(p-2)h^{\frac{p-2}{2}-2}h_\eps^{b-1}\left\langle\nabla u, \nabla h\right\rangle^{2} \Big)\eta^{2}e^{-f}\right)\notag\\
&\quad+2\int_\Omega h^{\frac{p}{2}-1}h_\eps^{b}\left\langle\nabla\eta, \nabla h\right\rangle\eta e^{-f}+2(p-2)\int_\Omega h^{\frac{p}{2}-2}h_\eps^{b}\left\langle\nabla u, \nabla h\right\rangle\left\langle\nabla u, \nabla\eta\right\rangle\eta f^{-f}\notag\\
&\quad+p\int_\Omega h^{\frac{p}{2}-1}h_\eps^{b}\left\langle\nabla u, \nabla h\right\rangle\eta^{2}e^{-f}+\frac{2}{m-1}\int_\Omega h^{\frac{p}{2}+1}h_\eps^{b}\eta^{2}e^{-f}\notag\\
&\leq 2(m-1)\kappa\int_\Omega h^{\frac{p}{2}}h_\eps^{b}\eta^{2} e^{-f}+\frac{2(p-1)}{m-1}\int_\Omega h^{\frac{p}{2}-1}h_\eps^{b}\eta^{2}e^{-f}.  \notag
\end{align}
Let 
$$ a_1=\begin{cases}
1\quad &\text{ if }p\geq 2\\
p-1&\text{ if }1<p<2
\end{cases}, $$
it is easy to see that 
$$ h^{\frac{p}{2}-1}h_\eps^{b-1}|\nabla h|^{2}+(p-2)h^{\frac{p-2}{2}-2}h_\eps^{b-1}\left\langle\nabla u, \nabla h\right\rangle^{2}\geq a_1 h^{\frac{p}{2}-1}h_\eps^{b-1}|\nabla h|^{2}. $$ 
Hence, by passing $\eps$ to $0$, we obtain
\begin{align}
ba_1&\int_\Omega h^{\frac{p}{2}+b-2}|\nabla h|^{2}\eta^{2}e^{-f}\notag\\
&\quad+2(p-2)\int_\Omega h^{\frac{p}{2}+b-2}\left\langle\nabla u, \nabla h\right\rangle\left\langle\nabla u, \nabla\eta\right\rangle\eta e^{-f}+2\int_\Omega h^{\frac{p}{2}+b-1}\left\langle\nabla h, \nabla\eta\right\rangle\eta e^{-f}\notag\\
&\quad+p\int_\Omega h^{\frac{p}{2}+b-1}\left\langle\nabla u, \nabla h\right\rangle\eta^{2}e^{-f}+\frac{2}{m-1}\int_\Omega h^{\frac{p}{2}+b+1}\eta^{2}e^{-f}\notag\\
&\leq 2(m-1)\kappa\int_\Omega h^{\frac{p}{2}+b}\eta^{2}e^{-f}+\frac{2(p-1)}{m-1}\int_\Omega h^{\frac{p}{2}+b-1}\left\langle\nabla u, \nabla h\right\rangle\eta^{2}e^{-f}. \label{e25}
\end{align}
Using \eqref{e25} and the argument as in \cite{WZ}, we can obtain the following lemma.
\begin{lemma}\label{l1}
Let $M$ be a smooth metric measure space with $Ric_f^{m}\geq-(m-1)\kappa,$ for some $\kappa\geq0$, and $\Omega\subset M$ is an open set and $v$ is a smooth weighted $p$-harmonic function on $M$. Let $u=-(p-1)\log v$ and $h=|\nabla u|^{2}$. Then for any $b>2$, there exist $c_1, c_2, c_3$ depending on $b, m, n$ such that
 \begin{align}
\int_\Omega \left|\nabla(h^{\frac{p}{4}+\frac{b}{2}}\eta)\right|^{2}e^{-f}+&c_1\int_\Omega h^{\frac{p}{2}+b+1}\eta^{2}e^{-f}\notag\\
&\leq \kappa c_2\int_\Omega h^{\frac{p}{2}+b}\eta^{2}e^{-f}+c_3\int_\Omega h^{\frac{p}{2}+b}|\nabla\eta|^{2}e^{-f},\label{e29}
\end{align}
for any $\eta\in{\cal C}_0^{\infty}(B_R)$, where $B_R$ is a geodesic ball centered at a fixed point $o\in M$. Moreover, we have $c_1\sim b, c_2\sim b$ (Here $c_1\sim b$ means $c_1$ is comparable to $b, c_2\sim b$ is understood the same way).
\end{lemma}
In \cite{BQ}, Bakry and Qian proved the following generalized Laplacian comparison theorem (also see Remark 3.2 in \cite{LXD})
\begin{equation}\label{laplaciancomparison}
 \Delta_f\rho:=\Delta-\left\langle\nabla f, \nabla\rho\right\rangle \leq (m-1)\ \sqrt[]{\kappa}\coth(\ \sqrt[]{\kappa}\rho) 
\end{equation}
provided that $Ric_f^m\geq-(m-1)\kappa$. Here $\rho(x):=dist(o,x)$ stands for the distance between $x\in M$ and a fixed point $o\in M$. This implies the volume comparison
\begin{equation}\label{volumecomparison}
 \frac{V_f(B_x(r_2))}{V_f(B_x(r_1))}\leq\frac{V_{\HH^m}(r_2)}{V_{\HH^m}(r_1)}
\end{equation}
(see \cite{Wu2} for details). It turns out that we have the local $f$-volume doubling property. Then we follow the Buser's proof \cite{Bus} or the Saloff-Coste's alternate proof (Theorem 5.6.5 in \cite{Sal}), we can easily get a local Neumann Poincar\'{e} inequality in the setting of smooth metric measure spaces. Using the volume comparison theorem, the local Neumann Poincar\'{e} inequality and following the argument in \cite{Sal1}, we obtain a local Sobolev inequality as belows. 
\begin{theorem}\label{sobolev}
Let $(M, g, e^{-f}d\mu)$ be an $n$-dimensional complete noncompact smooth metric measure space. If $\text{Ric}_f^{m}\geq-(m-1)\kappa$ for some nonnegative constants $\kappa$, then for any $p>2$, there exists a constant $c=c(n, p, m)>0$ depending only on $p, n, m$ such that 
$$ \left(\int_{B_R}|\varphi|^{\frac{2p}{p-2}}e^{-f}\right)^{\frac{p-2}{p}}\leq \frac{R^2.e^{c(1+\sqrt[]{\kappa}R)}}{V_f(B_R)^{\frac{2}{p}}}\int_{B_R}\left(|\nabla\varphi|^{2}+R^{-2}\varphi^{2}\right)e^{-f} $$
for any $\varphi\in C^\infty_0(B_p(R))$.
\end{theorem}
\begin{proof}
We refer the reader to \cite{Wu1} for the details of the argument.
\end{proof}
From now on, we suppose $\Omega=B_R$. Theorem \ref{sobolev} implies
\begin{align}
{~}&\left(\int_{B_R}h^{\frac{n(p/2+b)}{n-2}}\eta^{\frac{2n}{n-2}}e^{-f}\right)^{\frac{n-2}{n}}\notag\\
&\leq e^{c(1+\ \sqrt[]{\kappa}R)}{V_f(B_R)}^{-\frac{2}{n}}\left(R^{2}\int_{B_R} \left|\nabla(h^{\frac{p}{4}+\frac{b}{2}})\right|^{2}e^{-f}+\int_{B_R}h^{\frac{p}{2}+b}\eta^{2}e^{-f}\right)\label{e210}
\end{align}
where $c(n,p, m)>0$ depends only on $n, p$. Let $b_0=C(n,p, m)(1+\ \sqrt[]{\kappa}R)$ with $C(n,p, m)\geq c(n, p, m)$ large enough, then from \eqref{e29} and \eqref{e210} we infer
\begin{align}
{~}&\left(\int_{B_R}h^{\frac{n(p/2+b)}{n-2}}\eta^{\frac{2n}{n-2}}e^{-f}\right)^{\frac{n-2}{n}}+a_6be^{b_0}R^{2}{V_f(B_R)}^{-2/n}\int_{B_R}h^{\frac{p}{2}+b+1}\eta^{2}e^{-f}\notag\\
&\leq a_7b_0^{2}be^{b_0}{V_f(B_R)}^{-2/n}\int_{B_R}h^{\frac{p}{2}+b}\eta^{2}e^{-f}+a_8e^{b_0}{V_f(B_R)}^{-2/n}R^{2}\int_{B_R}h^{\frac{p}{2}+b}|\nabla\eta|^{2}e^{-f}.\label{e211}
\end{align}
This inequality and the Wang and Zhang's arguments (\cite{WZ}) imply the following lemma.
\begin{lemma}\label{l22}
Let $b_1=\left(b_0+\frac{p}{2}\right)\frac{n}{n-2}$. Then there exists $d=d(n, p ,m)>0$ such that
$$
||h||_{L^{b_1}(B_{3R/4})}\leq d\frac{b_0^{2}}{R^{2}}V_f(B_R)^{1/b_1}.
$$
\end{lemma}
\begin{proof}
The proof is followed by the argue in \cite{WZ}, hence, we omit the details.
\end{proof}
Now, we give a proof of the main theorem.
\begin{proof}[Proof of Theorem \ref{maintheorem}]
By \eqref{e211}, we have
\begin{align}
\left(\int_{B_R}h^{\frac{n(p/2+b)}{n-2}}\eta^{\frac{2n}{n-2}}e^{-f}\right)^{\frac{n-2}{n}}\leq a_{13}\frac{e^{b_0}}{{V_f(B_R)}^{2/n}}\int_{B_R}\left(b_0^{2}b\eta^{2}+R^{2}|\nabla\eta|^{2}\right)h^{\frac{p}{2}+b}e^{-f}. \label{e216}
\end{align}
In order to apply the Moser iteration, let us put 
$$ b_{k+1}=b_k\frac{n}{n-2}, \quad B_k=B\left(o, \frac{R}{2}+\frac{R}{4^k}\right), \quad k=1, 2, \ldots $$
and choose $\eta_k\in{\cal C}_0^{\infty}(B_R)$ such that 
$$ \eta_k\equiv1 \text{ in }B_{k+1}, \quad \eta\equiv0 \text{ in }B_R\setminus B_k, \quad |\nabla\eta_k|\leq\frac{C_14^{k}}{R}, \quad 0\leq\eta_k\leq1, $$
where $C_1$ is a certain constant. Hence in \eqref{e216}, by letting $b+\frac{p}{2}=b_k, \eta=\eta_k$, we obtain
$$ \left(\int_{B_{k+1}}h^{b_{k+1}}e^{-f}\right)^{\frac{1}{b_{k+1}}}\leq\left(a_{13}\frac{e^{b_0}}{{V_f(B_R)}^{2/n}}\right)^{\frac{1}{b_k}}\left(\int_{B_k}\left\{b_0^{2}b_k+R^{2}|\nabla\eta_k|^{2}\right\}h^{b_k}e^{-f}\right)^{\frac{1}{b_k}}. $$
By assumption of $|\nabla \eta_k|$, this implies
$$ ||h||_{L^{b_{k+1}}(B_{k+1})}\leq \left(a_{13}\frac{e^{b_0}}{{V_f(B_R)}^{2/n}}\right)^{\frac{1}{b_k}}(b_0^{2}b_k+16^{k})^{\frac{1}{b_k}}||h||_{L^{b_k}(B_k)}. $$
It is easy to see that $\suml_{k=1}^{\infty}\frac{1}{b_k}=\frac{n}{2b_1}$. The above inequality leads to 
\begin{align}
||h||_{L^{\infty}(B_{R/2})}
&\leq \left(a_{13}\frac{e^{b_0}}{{V_f(B_R)}^{2/n}}\right)^{\suml_{k=1}^{\infty}\frac{1}{b_k}}\prodl_{k=1}^{\infty}\left(b_0^{3}\left(\frac{n}{n-2}\right)^{k}+16^{k}\right)^{\frac{1}{b_k}}||h||_{L^{b_1}(B_{3R/4})}\notag\\
&\leq a_{14}\frac{e^{\frac{nb_0}{2b_1}}}{{V_f(B_R)}^{1/b_1}}b_0^{\frac{3n}{2b_1}}||h||_{L^{b_1}(B_{3R/4})}\label{e218}
\end{align}
here we used that $\suml_{k=1}^{\infty}\frac{k}{b_k}$ converges. Now, by Lemma \ref{l22} and \eqref{e218}, we conclude 
$$ ||h||_{L^{\infty}(B_{R/2})}\leq a_{15}\frac{b_0^{2}}{R^{2}}. $$
The proof is complete.
\end{proof}

As a consequence, we obtain the folowing important theorem ralating to Liouville-property for weighted $p$-harmonic functions.
\begin{theorem}\label{main}
Assume that $(M, g)$ is a smooth metric measure space with $\text{Ric}_f^{m}\geq 0$. If $u$ is a weighted $p$-harmonic function bounded from below on $M$ and $u$ is of sublinear growth then $u$ is constant. 
\end{theorem} 
Moreover, let $x, y\in M$ be arbitrary points. There is a minimal geodesic $\gamma(s)$ joining $x$ and $y$, $\gamma: [0,1]\to M, \gamma(0)=x, \gamma(1)=y$. By integrating \eqref{maine} over this geodesic, we obtain the below Harnack inequality.
\begin{theorem}
Let $(M, g, e^{-f}dv)$ be a complete smooth metric measure space of dimension $n\geq2$ with $Ric_f^m\geq-(m-1)\kappa$. Suppose that $v$ is a positive weighted $p$-harmonic function on the geodesic ball $B(o, R)\subset M$. There exists a constant $C_{p,n,m}$ depending only on $p,n,m$ such that 
$$ v(x)\leq e^{C_{p,n,m}(1+\ \sqrt[]{\kappa}R)}v(y), \quad\quad \forall x, y \in B(o,R/2). $$
If $\kappa=0$, we have a uniform constant $c_{p,n,m}$ (independent of $R$) such that
$$ \sup\limits_{B(o,R/2)}v\leq c_{p,n}\inf\limits_{B(o,R/2)}v. $$ 
\end{theorem}
We finish this section by giving the following gradient estimates for weighted $p$-eigenfunctions with repect to the first eigenvalue $\lambda_{1, p}$. This estimate will be used in the next section where we show the sharpness of our gradient estimates for weighted $p$-eigenfunctions.
\begin{proposition}\label{ddmain}
Assume that $(M, g)$ is a smooth metric measure space with $\text{Ric}_f^{m}\geq -(m-1)\kappa$ for some $\kappa\geq0$. If $v$ is a positive weighted $p$-eigenfunction associated to the first eigenvalue $\lambda_{1, p}$ then $|\nabla(\ln v)|$ is bounded. 
\end{proposition}
\begin{proof}Theorem \ref{maintheorem} implies
$$ \frac{|\nabla v|}{v}\leq \frac{C(1+\ \sqrt[]{\kappa}R)}{R} \quad \text{ on }\quad B(o, R/2).$$
Letting $R\to\infty$, we are done. 
\end{proof}
\section{Global sharp gradient estimates for weighted $p$-Laplacian}
\setcounter{equation}{0}
Recall that a function $v$ is an eigenfunction of $p$-Laplacian with corresponding eigenvalue $\lambda_{1,p}\geq0$ if 
\begin{equation}\label{s0} 
e^{f}div(e^{-f}|\nabla v|^{p-2}\nabla v)=-\lambda_{1,p}|v|^{p-2}v. 
\end{equation}
In this section, we only consider positive solution $v$. Set $u=-(p-1)\ln v$, the equation \eqref{s0} can be rewritten as follows
\begin{equation}\label{sehs}
e^{f}div(e^{-f}|\nabla u|^{p-2}\nabla u)=|\nabla u|^{p}+\lambda_{1,p}(p-1)^{p-1}.
\end{equation}  
Put $h:=|\grad u|^2$, assume that $h>0$. As in \cite{SW}, we consider 
$$ \mathfrak{L}(\psi):= e^fdiv\Big(e^{-f}h^{p/2-1}A(\nabla \psi)\Big)$$
which is a slight modification of ${\cal L}(\psi)$. By Lemma \ref{d2}, we have that
$$\begin{aligned}
\mathfrak{L}(h)=\quad&2h^{p/2-1}(u_{ij}^2+\text{Ric}_f(\nabla u, \nabla u))+\left(\frac{p}{2}-1\right)h^{p/2-2}|\nabla h|^{2}\\
&+ ph^{\frac{p}{2}-1}\left\langle \nabla u, \nabla h\right\rangle 
\end{aligned}$$

Let $\left\{ e_1, e_2, \ldots, e_n\right\}$ be an orthonormal frame on $M$ with $|\nabla u|e_1=\nabla u$. Then \eqref{sehs} can be read as 
$$ (p-1)u_{11}+\suml_{i=2}^nu_{ii}=h+\left\langle\nabla f, \nabla u\right\rangle +|\nabla u|^{2-p}\lambda_{1,p}(p-1)^{p-1}. $$
Therefore, 
$$ \begin{aligned}
u_{ij}^2
&\geq \suml_{i=1}^nu_{1i}^2+\frac{1}{n-1}\left(\suml_{i=2}^nu_{ii}\right)^2\\
&\geq \suml_{i=1}^nu_{1i}^2+\frac{(h+\lambda_{1,p}(p-1)^{p-1}|\nabla u|^{2-p}+\left\langle\nabla f, \nabla u\right\rangle -(p-1)u_{11})^{2}}{n-1}\\
&\geq \suml_{i=1}^nu_{1i}^2+\frac{1}{n-1}\left(\frac{(h+\lambda_{1,p}(p-1)^{p-1}|\nabla u|^{2-p}-(p-1)u_{11})^{2}}{1+\frac{m-n}{n-1}}-\frac{\left\langle\nabla u, \nabla f\right\rangle^{2} }{\frac{m-n}{n-1}}\right)\\
&\geq \suml_{i=1}^nu_{1i}^2+\frac{(h+\lambda_{1,p}(p-1)^{p-1}|\nabla u|^{2-p})^{2}}{m-1}-\frac{\left\langle\nabla f, \nabla u\right\rangle^{2} }{m-n}\\
&\quad \quad -\frac{2(p-1)}{m-1}hu_{11}-\frac{2\lambda_{1,p}(p-1)^{p}}{m-1}|\nabla u|^{2-p}u_{11}.
\end{aligned} $$
Note that we used $(a-b)^{2}\geq\frac{a^2}{1+\delta}-\frac{b^{2}}{\delta}$ for $\delta >0$ in the fourth inequality. Again, by using the identities
$$ 2hu_{11}=\left\langle \nabla u, \nabla h\right\rangle, \quad \suml_{i=1}^nu_{1i}^2=\frac{1}{4}\frac{|\nabla h|^2}{h}  $$
we conclude that
$$\begin{aligned}
u_{ij}^2\geq& \frac{1}{4}\frac{|\nabla h|^2}{h}+ \frac{(h+\lambda_{1,p}(p-1)^{p-1}|\nabla u|^{2-p})^{2}}{m-1}-\frac{\left\langle\nabla f, \nabla u\right\rangle^{2} }{m-n}\\
&\quad \quad -\frac{2(p-1)}{m-1}\left\langle\nabla h, \nabla u\right\rangle -\frac{2\lambda_{1,p}(p-1)^{p}}{m-1}|\nabla u|^{2-p}\frac{\left\langle\nabla h, \nabla u\right\rangle }{h}.
\end{aligned}$$
Assume that $\text{Ric}_f^{m}\geq-(m-1)\kappa$, we infer
 \begin{align}
{\mathfrak L}(h)
=&2h^{p/2-1}(u_{ij}^2+\text{Ric}_f(\nabla u, \nabla u))+\left(\frac{p}{2}-1\right)h^{p/2-2}|\nabla h|^{2}\notag\\
&\quad+ ph^{\frac{p}{2}-1}\left\langle \nabla u, \nabla h\right\rangle \notag\\
\geq& 2h^{p/2-1}\left(\frac{1}{4}\frac{|\nabla h|^2}{h}+ \frac{(h+\lambda_{1,p}(p-1)^{p-1}|\nabla u|^{2-p})^{2}}{m-1}+\left(\text{Ric}_f-\frac{df\otimes df}{m-n}\right)(\nabla u, \nabla u)\right)\notag\\
&\quad\quad+\left(\frac{p}{2}-1\right)h^{p/2-2}|\nabla h|^{2}+ ph^{\frac{p}{2}-1}\left\langle \nabla u, \nabla h\right\rangle\notag \\
&\quad\quad -\frac{4(p-1)}{m-1}h^{\frac{p}{2}-1}\left\langle\nabla h, \nabla u\right\rangle -\frac{2\lambda_{1,p}(p-1)^{p}}{m-1}\frac{\left\langle\nabla h, \nabla u\right\rangle }{h}\notag\\
\geq& 2h^{p/2-1}\left(\frac{(h+\lambda_{1,p}(p-1)^{p-1}|\nabla u|^{2-p})^{2}}{m-1}-(m-1)h\right)\notag\\
&\quad\quad+\frac{p-1}{2}h^{p/2-2}|\nabla h|^{2}+ ph^{\frac{p}{2}-1}\left\langle \nabla u, \nabla h\right\rangle\notag \\
&\quad \quad -\frac{4(p-1)}{m-1}h^{\frac{p}{2}-1}\left\langle\nabla h, \nabla u\right\rangle -\frac{2\lambda_{1,p}(p-1)^{p}}{m-1}\frac{\left\langle\nabla h, \nabla u\right\rangle }{h}.\label{s2}
 \end{align} 
To show the sharp estimate, let $x$ be the unique positive root of the equation 
$$ x^{\frac{p}{2}}-(m-1)x^{\frac{p-1}{2}}+\lambda_{1,p}(p-1)^{p-1}=0. $$
For any $\delta>0$, we consider 
$$ \omega=\begin{cases}
h-(x+\delta), \quad &h>x+\delta\\
0,&\text{ otherwise}
\end{cases}. $$
To show global sharp estimate of weighted $p$-eigenfunction, we need to have a upper bound of $
\lambda_{1, p}$ as follows
\begin{lemma}\label{eigenvalue}
Let $(M^n, g, e^{-f}dv)$ be an $n$-dimensional complete noncompact manifold with $Ric_f^{m}\geq-(m-1)$. Then 
$$ \lambda_{1,p}\leq\left(\frac{m-1}{p}\right)^p. $$
\end{lemma}
In order to prove lemma \ref{eigenvalue}, let us recall a definition.
\begin{define}(see \cite{BK})
Let $(M^n, g, e^{-f}dv)$ be a smooth metric measure space. For a fixed point $o\in M$, let $\overline{B}_o(r)=\left\{q\in M: dist(o, q)\leq r\right\}$. An end $E$ is an unbounded component $E$ of $M\setminus\overline{B}_o(r_0)$ for some $r_0\geq0$. For any $1\leq p<\infty$. The end E is said to be $p$-parabolic if for each $K\Subset M$ and $\eps>0$ there exists a Lipschitz function $\phi$ with compact support, $\phi\geq1$ on $K$, such that $\int_Eg_\phi^p<\eps$. Otherwise, $E$ is $p$-nonparabolic. Here $g_\phi(x)$ is defined as
$$
 g_\phi(x)=\liminf\limits_{r\to 0^+}\sup\limits_{ y\in B_x(r)}\frac{dist(\phi(y), \phi(x))}{dist(y,x)}.
$$
\end{define}
\begin{proof}[Proof of lemma \ref{eigenvalue}]
Without loss of generality, we may assume that $\lambda_{1,p}$ is positive. By the variational characterization of $\lambda_{1,p}$, we know that $M$ has infinite $f$-volume. The theorem 0.1 in \cite{BK} implies $M$ is $p$-nonparabolic, moreover 
$$ V_f(B(r))\geq Ce^{p\lambda_{1,p}^{1/p}r}, $$
for all sufficiently large $r$ and $C$ is a constant dependent on $r$. On the other hand, the volume comparison theorem \eqref{volumecomparison} infers
$$ V_f(B(r))\leq C_1e^{(m-1)r}. $$
Here $C_1$ is a constant dependent on $r$. Therefore, we obtain
$$ Ce^{p\lambda_{1,p}^{1/p}r}\leq C_1e^{(m-1)r}, $$
or equivalently,
$$ \lambda_{1,p}^{1/p}\leq\frac{1}{pr}\ln\left(\frac{C_1}{C}\right)+\frac{m-1}{p} $$
for all sufficiently large $r$. Lettting $r\to \infty$, we have
$$ \lambda_{1,p}\leq\left(\frac{m-1}{p}\right)^p. $$
The proof is complete. 
\end{proof}
We have following key lemma.
\begin{lemma}\label{lsehb}
Let $(M^n, g, e^{-f}dv)$ be an $n$-dimensional complete noncompact manifold with $Ric_f^{m}\geq-(m-1)$. 
Then there are some positive constants $a$ and $b$ depending on $p, n, m$ and $\delta$ such that
\begin{equation}\label{sehb}
 \mathfrak{L}(\omega)\geq a\omega-b|\nabla\om|,
\end{equation}
in the weak sense, namely
\begin{equation}\label{e24}
\int_M  \mathfrak{L}(\phi)\omega e^{-f}\geq \int_M\phi(a\om-b|\nabla\om|)e^{-f}
\end{equation}
for any non-negative function $\phi$ with compact support on $M$.
\end{lemma}
\begin{proof}Since $h=|\nabla u|^2$, by Proposition \ref{ddmain}, we have 
$$ x+\delta\leq h\leq c(n,p,m). $$
Denote by $\Omega=\left\{h\geq x+\delta\right\}$, then \eqref{s2} implies that there are positive constants $c_1, c_2$ depending only on $n,p,m$ such that on $\Omega$
\begin{equation}\label{s28}
\mathfrak{L}(\omega)\geq c_1\left(h^\frac{p}{2}-(m-1)h^\frac{p-1}{2}+\lambda_{1,p}(p-1)^{p-1}\right)-c_2|\nabla\omega|.
\end{equation}
Now, we can follow an argument in \cite{SW} to prove that on $\Omega$
\begin{equation}\label{s29}
h^\frac{p}{2}-(m-1)h^\frac{p-1}{2}+\lambda_{1,p}(p-1)^{p-1}\geq c_3\omega
\end{equation}
for some positive constant $c_3$ depending only on $n,p,m,\delta$. Indeed, we consider both sides of \eqref{s29} as a function of $h$. It is easy to see that \eqref{s29} is valid when $h=x$ for any choice of $c_3$. Now, we view the left hand side of \eqref{e29} as a function of $h$, its derivative is 
$$ \frac{p}{2}h^{\frac{p}{2}-1}-\frac{(m-1)(p-1)}{2}h^\frac{p-3}{2}=\frac{1}{2}h^\frac{p-3}{2}\left(ph^\frac{1}{2}-(p-1)(n-1)\right). $$
By lemma \ref{eigenvalue}, we have $\lambda_{1,p}\leq\left(\frac{m-1}{p}\right)^p$, this implies $x\geq(p-1)^2\left(\frac{m-1}{p}\right)^2$. Since $h\geq x+\delta$, we conclude that 
$$ ph^\frac{1}{2}-(p-1)(m-1)\geq c(n,p,m,\delta)>0. $$
Hence, \eqref{s29} holds true on $\Omega$ for some $0<c_3<c(n,p,m, \delta)$. 

From \eqref{s28}, \eqref{s29}, we obtain 
\begin{equation}\label{d2d} \mathfrak{L}(\omega)\geq a\omega-b|\nabla\omega| 
\end{equation}
on $\Omega$. Here $a, b$ are positive constants depending only on $p,n,m,\delta$.

Using the integration by parts, we have 
$$ \begin{aligned}
\int_M\mathfrak{L}(\phi)\omega e^{-f}
&=\int_\Omega\mathfrak{L}(\phi)\omega e^{-f}\\
&=\int_M\phi\mathfrak{L}\omega e^{-f}+\int_{\partial\Omega}\left\langle h^{\frac{p}{2}-1}A(\nabla\phi), \nu\right\rangle \omega e^{-f}-\int_{\partial\Omega}\left\langle h^{\frac{p}{2}-1}A(\nabla\omega), \nu\right\rangle\phi e^{-f}.  
\end{aligned} $$
Here $\nu$ is the outward unit normal vector of $\partial\Omega$. Since $\nu=-\frac{\nabla h}{|\nabla h|}=-\frac{\nabla\omega}{|\nabla\omega|}$ and $\omega=0$ on $\partial\Omega$, this implies
\begin{align}
\int_M\mathfrak{L}(\phi)\omega e^{-f}
&=\int_\Omega\phi\mathfrak{L}(\omega)e^{-f}+\int_{\partial\Omega}h^{\frac{p}{2}-1}\frac{\phi\left\langle A(\nabla\omega), \nabla\omega\right\rangle }{|\nabla\omega|}e^{-f}\notag\\
&\geq\int_\Omega \phi\mathfrak{L}(\omega)e^{-f}.\notag\\
&\geq\int_\Omega\phi(a\omega-b|\nabla\omega|)e^{-f}\notag\\
&=\int_M\phi(a\omega-b|\nabla\omega|)e^{-f}.\notag
\end{align}
where we used \eqref{d2d} in the third inequality. The proof is complete.
\end{proof} 
Now, we give a proof of the theorem \ref{2ndtheorem}.
\begin{proof}[Proof of Theorem \ref{2ndtheorem}]We follow the proof in \cite{SW}. First, we will prove that $\om\equiv0$. Indeed, for any cut-off function $\phi$ on $M$, and for any $q>0$, by using \eqref{e24}, we have
$$ \int_M\omega\mathfrak{L}(\phi^2\omega^q)e^{-f}\geq\int_M(a\phi^2\omega^{q+1}-b\phi^2\omega^q|\nabla\omega|)e^{-f}. $$
Integration by parts implies
$$ \int_M\omega\mathfrak{L}(\phi^2\omega^q)e^{-f}=-\int_\Omega\left\langle A(\nabla(\phi^2\omega^q)), \nabla\omega\right\rangle h^{\frac{p}{2}-1}e^{-f}.  $$
Therefore,
$$\begin{aligned}
 a\int_M\phi^2\omega^{q+1}e^{-f}
\leq b&\int_\Omega\phi^2\omega^q|\nabla\omega|e^{-f}+C(n,p,m\delta)\int_\Omega\phi|\nabla\phi||\nabla\omega|\omega^qe^{-f}\\
&-\int_\Omega q\phi^2\omega^{q-1}\left\langle A(\nabla\omega), \nabla\omega\right\rangle e^{-f}. 
\end{aligned} $$
It is easy to see that 
$$\begin{aligned}
 \left\langle A(\nabla\omega), \nabla\omega\right\rangle 
&=|\nabla\omega|^2+(p-2)\frac{\left\langle\nabla u, \nabla\omega\right\rangle }{|\nabla u|^2}\\
&\geq(p-1)|\nabla\omega|^2.
\end{aligned} $$
Hence, for any $\eps>0$, we have
$$ \begin{aligned}
a\int_M\phi^2\omega^{q+1}e^{-f}
\leq b\eps&\int_\Omega\phi^2\omega^{q+1}e^{-f}+\frac{b}{4\eps}\int_\Omega\phi^2\omega^{q-1}|\nabla\omega|^2e^{-f}+\eps\int_\Omega|\nabla\phi|^2\omega^{q+1}e^{-f}\\
&+\frac{\widehat{c}}{4\eps}\int_\Omega\phi^2\omega^{q-1}|\nabla\omega|^2e^{-f}-\widetilde{c}\int_\Omega\phi^2\omega^{q-1}|\nabla\omega|^2e^{-f},
\end{aligned} $$
where $\widehat{c}$ and $\widetilde{c}$ are constants depending only on $n,p,m,\delta$. Choose $q$ such that $b+\widehat{c}=4\eps\widetilde{c}$ then
$$ (a-b\eps)\int_M\phi^2\omega^{q+1}e^{-f}\leq\eps\int_M|\nabla\phi|^2\omega^{q+1}e^{-f}. $$
Now a standard argument implies either $\omega\equiv0$ or for all $R\geq1$,
\begin{equation}\label{contra} 
\int_{B(R)}\omega^{q+1}e^{-f}\geq c_1e^{R\ln\frac{c_2}{\eps}}
\end{equation}
for some positive constants $c_1$ and $c_2$ independent of $\eps$. 

Since $\omega$ is bounded and the $f$-volume of the ball $B(R)$ satisfies $V_f(B(R))\leq ce^{(m-1)R}$ (by \eqref{volumecomparison}), if $\eps>0$ is chosen sufficiently small, \eqref{contra} can not hold. Hence, $\omega\equiv0$. 
This implies $h\leq x$ since $\delta$ is arbitrary. Thus, $|\nabla\ln v|\leq y$. The proof is complete.
\end{proof}
Let $\lambda_{1,p}=\left(\frac{m-1}{p}\right)^p$, it is easy to see that the equation 
$$ (p-1)y^p-(m-1)y^{p-1}+\lambda_{1,p}=0 $$
has the unique positive solution $y=\frac{m-1}{p}$. Hence, we have the following corollary.
\begin{corollary}\label{cor23}
Let $(M^n, g, e^{-f})$ be an $n$-dimensional complete noncompact manifold with $Ric_f^m\geq-(m-1)$. Suppose that $u$ is a positive solution of 
$$ e^{f}div\left(e^{-f}|\nabla u|^{p-2}\nabla u\right)=-\left(\frac{m-1}{p}\right)^pu^{p-1}. $$
Then 
$$ \frac{|\nabla u|}{u}\leq\frac{m-1}{p}. $$ 
\end{corollary}
\section{Rigidity of manifolds with maximal $\lambda_{1,p}$}
\setcounter{equation}{0}
In this section, we study structure at infinity of manifolds with maximal $\lambda_{1,p}$. Our main purpose is to prove the theorem \ref{3rdtheorem} stated in the introduction part.
\begin{theorem}\label{theorem1}
Let $(M^n, g, e^{-f}dv)$ be a smooth metric measure space of dimension $n\geq2$. Suppose that $Ric_f^m\geq-(m-1)$ and $\lambda_{1,p}=\left(\frac{m-1}{p}\right)^p$. Then either $M$ has no $p$-parabolic ends or $M=\RR\times N^{n-1}$ for some compact manifold $N$.
\end{theorem} 
\begin{proof}Our argument is close to the argument in \cite{SW}. 
Suppose that $M$ has a $p$-parabolic end $E$. Let $\beta$ be the Busemann function associated with a geodesic ray $\gamma$ contained in $E$, namely,
$$ \beta(x)=\lim\limits_{t\to\infty}(t-dist(x, \gamma(t))). $$
Using the Laplacian comparison theorem \eqref{laplaciancomparison}, we have 
$$ \Delta_f\beta\geq-(m-1). $$
Hence,
$$ \begin{aligned}
\Delta_{p,f}\left(e^{\frac{m-1}{p}\beta}\right)
&=e^fdiv\left(e^{-f}\left(\frac{m-1}{p}\right)^{p-2}e^{\frac{m-1}{p}(p-2)\beta}\nabla e^{\frac{m-1}{p}\beta}\right)\\
&=\left(\frac{m-1}{p}\right)^{p-1}e^{\frac{m-1}{p}(p-1)\beta}\Delta_f\beta+\left(\frac{m-1}{p}\right)^{p-1}\left(\frac{m-1}{p}\right)(p-1)e^{\frac{m-1}{p}(p-1)\beta}|\nabla\beta|^2\\
&\geq(m-1)\left(\frac{m-1}{p}\right)^{p-1}e^{\frac{m-1}{p}(p-1)\beta}\left(\frac{p-1}{p}-1\right)\\
&=-\left(\frac{m-1}{p}\right)^pe^{\frac{m-1}{p}(p-1)\beta}.
\end{aligned} $$
Therefore, let $\omega:=e^{\frac{m-1}{p}\beta}$, we obtain 
$$ \Delta_{p,f}(\omega)\geq-\lambda_{1,p}\omega^{p-1}. $$
Suppose that $\phi$ is a nonnegative compactly supported smooth function on $M$. Then by the variational principle,
$$ \lambda_{1,p}\int_M(\phi\omega)^pe^{-f}\leq\int_M|\nabla(\phi\omega)|^pe^{-f}. $$
Noting that, integration by parts implies
$$ \int_M\phi^p\omega\Delta_{p,f}(\omega) e^{-f}=-\int_M\phi^p|\nabla\omega|e^{-f}-p\int_M\phi^{p-1}\omega\left\langle\nabla\phi,\nabla\omega\right\rangle|\nabla\omega|^{p-2}e^{-f}  $$
and
$$ \begin{aligned}
|\nabla(\phi\omega)|^p
&=\left(|\nabla\phi|^2\omega^2+2\phi\omega\left\langle\nabla\phi, \nabla\omega\right\rangle+\phi^2|\nabla\omega|^{p-2} \right)^\frac{p}{2}\\
&\leq \phi^p|\nabla\omega|^p+p\phi\omega\left\langle\nabla\phi, \nabla\omega\right\rangle\phi^{p-2}|\nabla\omega|^{p-2}+c|\nabla\phi|^2\omega^p 
\end{aligned} $$
for some constant $c$ depending only on $p$, we infer
\begin{align}
&\int_M\phi^p\omega(\Delta_{p,f}(\omega)+\lambda_{1,p}\omega^{p-1})e^{-f}\notag\\
&=\lambda_{1,p}\int_M(\phi\omega)^pe^{-f}-\int_M\phi^p|\nabla\omega|^pe^{-f}-p\int_M\phi^{p-1}\omega\left\langle\nabla\phi,\nabla\omega\right\rangle|\nabla\omega|^{p-2}e^{-f}\notag\\
&\leq\int_M|\nabla(\phi\omega)|^pe^{-f}-\int_M\phi^p|\nabla\omega|^pe^{-f} -p\int_M\phi^{p-1}\omega\left\langle\nabla\phi,\nabla\omega\right\rangle|\nabla\omega|^{p-2}e^{-f}\notag\\
&\leq c\int_M|\nabla\phi|^2\omega^pe^{-f}.\label{swe31}
\end{align} 
Now, we choose
$$ \phi=\begin{cases}
1&\quad\text{in }B(R)\\
0,&\quad\text{ on }M\setminus B(2R)
\end{cases} $$
such that $|\nabla\phi|\leq\frac{2}{R}$. Then we conclude
\begin{align}
\int_M|\nabla\phi|^2\omega^pe^{-f}
&=\int_M|\nabla\phi|^2e^{(m-1)\beta}e^{-f}\notag\\
&\leq\frac{4}{R^2}\int_{B(2R)\setminus B(R)}e^{(m-1)\beta}e^{-f}\notag\\
&=\frac{4}{R^2}\int_{E\cap(B(2R)\setminus B(R))}e^{(m-1)\beta}e^{-f}+\frac{4}{R^2}\int_{(M\setminus E)\cap(B(2R)\setminus B(R))}e^{(m-1)\beta}e^{-f}.\label{swe32}
\end{align}
Since $\lambda_{1,p}=\left(\frac{m-1}{p}\right)^p$, by theorem 0.1 in \cite{BK}, it turns out that
$$ V_f(E\setminus B(R))\leq  \widehat{c}e^{-(m-1)R}. $$
Hence, the first term of \eqref{swe32} tends to $0$ as $R$ goes to $\infty$. On the other hand, by \cite{LW3} we have 
$$ \beta(x)\leq-r(x)+\widetilde{c} $$
on $M\setminus E$. The volume comparison \eqref{volumecomparison} implies $V_f(B(R))\leq c'e^{(m-1)R}$. It turns out that the second term of \eqref{swe32} also goes to $0$ as $R\to\infty$. Therefore, \eqref{swe31} infers
$$ \Delta_{p,f}(\omega)+\lambda_{1,p}\omega^{p-1}\equiv0. $$
This implies 
$$ \Delta_f\beta=-(m-1). $$
Moreover, all the inequalities used to prove $\Delta_f\beta\geq-(m-1)$ become equalities (see Theorem 1.1 in \cite{LXD}). By the proof of the Theorem 1.1 in \cite{LXD} and the argument in \cite{LW3}, we conclude that $M=\RR\times N^{n-1}$ for some compact manifold $N$ of dimension $n$. The proof is complete.
\end{proof}
For $p$-nonparabolic end, we have the following theorem.
\begin{theorem}\label{nonparabolic}
Let $(M^n, g, e^{-f}dv)$ be a smooth metric measure space of dimension $n\geq3$. Suppose that $Ric_f^m\geq-(m-1)$ and $\lambda_{1,p}=\left(\frac{m-1}{p}\right)^p$ for some $2\leq p\leq\frac{(m-1)^2}{2(m-2)}$. Then either $M$ has only one $p$-nonparabolic end or $M=\RR\times N^{n-1}$ for some compact manifold $N$.
\end{theorem}
To prove theorem \ref{nonparabolic}, let us recall a fact on weighted Poincar\'{e} inequality in \cite{Liwang}.
\begin{proposition}[\cite{Liwang}, Proposition 1.1]\label{liwa}
Let $M$ be a complete Riemannian manifold. If there exists a nonnegative function $h$ defined on $M$, that is not identically $0$, satisfying
$$ \Delta h(x)+\left\langle\nabla g, \nabla h\right\rangle(x) \leq - \rho(x)h(x), $$
for some nonnegative function $\rho$, then the weighted Poincar\'{e} inequality
$$ \int_M\rho(x)\phi^2(x)e^{g(x)}\leq \int_M|\nabla\phi|^2(x)e^{g(x)} $$
must be hold true for all compactly supported smooth function $\phi\in{\cal C}_0^\infty(M)$. 
\end{proposition}
Now we give a proof of theorem \ref{nonparabolic}.
\begin{proof}[Proof of theorem \ref{nonparabolic}]
If $p=2$ then we have $\lambda_1(M)=\lambda_{1,2}=\frac{(m-1)^2}{4}$. Hence, by Theorem 1.5 in \cite{Wu2}, we are done. Therefore, we may assume $p>2$.  

Now let $u$ be a positive weighted $p$-eigenfunction of the weighted $p$-Laplacian satisfying
$$ \Delta_{p,f}u=e^{f}div(e^{-f}|\nabla u|^{p-2}\nabla u)=-\lambda_{1,p}u^{p-1}. $$
This implies
$$ \Delta u+\left\langle-\nabla f+(p-2)\frac{\nabla|\nabla u|}{|\nabla u|}, \nabla u\right\rangle=-\lambda_{1,p}\frac{u^{p-2}}{|\nabla u|^{p-2}}u.  $$
Let 
$$ g=-f+(p-2)\ln|\nabla u|, \quad \text{and }\quad \rho=\lambda_{1,p}\frac{u^{p-2}}{|\nabla u|^{p-2}}. $$
By proposition \ref{liwa}, we obtain
$$ \int_M\lambda_{1,p}\frac{u^{p-2}}{|\nabla u|^{p-2}}\phi^2|\nabla u|^{p-2}e^{-f}\leq\int|\nabla\phi|^2|\nabla u|^{p-2}e^{-f} $$
or equivalently,
$$ \lambda_{1,p}\int_Mu^{p-2}\phi^2e^{-f}\leq\int_M|\nabla u|^{p-2}|\nabla\phi|^2e^{-f},
 $$
for any compactly supported smooth function $\phi$ on $M$. Therefore
$$\begin{aligned}
 \lambda_{1,p}\int_M\phi^2e^{-f}
&=\lambda_{1,p}\int_M\left(\phi u^{-\frac{p-2}{2}}\right)^2u^{p-2}e^{-f}\\
&\leq\int_M\left|\nabla\left(\phi u^{-\frac{p-2}{2}}\right)\right|^2|\nabla u|^{p-2}e^{-f}
\end{aligned} $$
By a direct computation, we have
 \begin{align}
\int_M&\left|\nabla\left(\phi u^{-\frac{p-2}{2}}\right)\right|^2|\nabla u|^{p-2}e^{-f}\notag\\
=&\int_M|\nabla\phi|^2\frac{|\nabla u|^{p-2}}{u^{p-2}}e^{-f}+\phi^2\left|\nabla\left(u^{-\frac{p-2}{2}}\right)\right|^2|\nabla u|^{p-2}e^{-f}\notag\\
&\quad +2\int_M\phi u^{-\frac{p-2}{2}}\left\langle\nabla\phi, \nabla\left(u^{-\frac{p-2}{2}}\right)\right\rangle |\nabla u|^{p-2}e^{-f} \notag\\
=&\int_M|\nabla\phi|^2\frac{|\nabla u|^{p-2}}{u^{p-2}}e^{-f}+\left(\frac{p-2}{2}\right)^2\int_M\frac{|\nabla u|^{p}}{u^p}\phi^2e^{-f}+\frac{1}{2}\int_M\left\langle\nabla\phi^2, \nabla u^{-(p-2)}\right\rangle|\nabla u|^{p-2}e^{-f}.\label{claim0} 
\end{align} 
We claim that
\begin{equation}\label{claim1}
\frac{1}{2}\int\left\langle\nabla\phi^2, \nabla v^{-2}\right\rangle|\nabla u|^{p-2}e^{-f}=-\frac{p-2}{2}\int_M\left(\lambda_{1,p}+(p-1)\frac{|\nabla u|^p}{u^p}\right)\phi^2e^{-f}. 
\end{equation}
Suppose that the claim \eqref{claim1} is verified. By \eqref{claim0}, \eqref{claim1}, we obtain
$$ \begin{aligned}
 \lambda_{1,p}\int_M\phi^2e^{-f}
\leq& \int_M|\nabla\phi|^2\frac{|\nabla u|^{p-2}}{u^{p-2}}e^{-f}+\left(\frac{p-2}{2}\right)^2\int_M\frac{|\nabla u|^{p}}{u^p}\phi^2e^{-f}\\
&\quad -\frac{p-2}{2}\int_M\left(\lambda_{1,p}+(p-1)\frac{|\nabla u|^p}{u^p}\right)\phi^2e^{-f}. 
\end{aligned} $$
This means
$$ \frac{p}{2}\lambda_{1,p}\int_M \phi^2e^{-f}+\frac{p(p-2)}{4}\int_M\phi^2\frac{|\nabla u|^p}{u^p}e^{-f}\leq\int_M\frac{|\nabla u|^{p-2}}{u^{p-2}}|\nabla\phi|^2e^{-f}.$$
Since $\lambda_{1,p}=\left(\frac{m-1}{p}\right)^p$, the corollary \ref{cor23} implies $\frac{\nabla u}{u}\leq\frac{m-1}{p}$. Therefore,
$$ \frac{(m-1)^2}{2p}\int_M\phi^2e^{-f}\leq\int_M|\nabla\phi|^2e^{-f}. $$
Hence, $\lambda_1(M)\geq\frac{(m-1)^2}{2p}$. Here $\lambda_1(M)$ is the first eigenvalue of the weighted Laplacian. Since $p\leq\frac{(m-1)^2}{2(m-2)}$, we have $\lambda_{1}(M)\geq (m-2)$. By a Wang's theorem (see \cite{W1}), we are done.

The rest of the proof is to verify the claim \eqref{claim1}. Indeed, we have
\begin{align}
\frac{1}{2}&\int_M\left\langle\nabla\phi^2, \nabla u^{-(p-2)}\right\rangle|\nabla u|^{p-2}e^{-f}\notag\\
&=-\frac{1}{2}\int_M\Delta_f(u^{-(p-2)})|\nabla u|^{p-2}\phi^2e^{-f}-\frac{1}{2}\int_M\left\langle\nabla u^{-(p-2)}, \nabla|\nabla u|^{p-2}\right\rangle\phi^2e^{-f}.\label{proof1}
\end{align} 
Note that 
$$ e^fdiv\left(e^{-f}|\nabla u|^{p-2}\nabla u\right)=-\lambda_{1,p}u^{p-1} $$
we have
$$ |\nabla u|^{p-2}\Delta_fu+\left\langle\nabla|\nabla u|^{p-2}, \nabla u\right\rangle=-\lambda_{1,p}u^{p-1}.  $$
Therefore,
 \begin{align}
|\nabla u|^{p-2}&\Delta_f(u^{2-p})+\left\langle\nabla(u^{2-p}, \nabla(|\nabla p|^{p-2}))\right\rangle\notag\\
=& (2-p)|\nabla u|^{p-2}u^{1-p}\Delta_fu+(2-p)(1-p)|\nabla u|^{p-2}u^{-p}|\nabla u|^2\notag\\
&\quad +(2-p)u^{1-p}\left\langle\nabla u, \nabla(|\nabla u|^{p-2})\right\rangle\notag\\
&=(2-p)u^{1-p}(-\lambda_{1,p}u^{p-1})+(2-p)(1-p)\frac{|\nabla u|^p}{u^p}=(p-2)\left(\lambda_{1,p}+(p-1)\frac{|\nabla u|^p}{u^p}\right). \label{proof2} 
\end{align} 
By \eqref{proof1} and \eqref{proof2}, the claim \ref{claim1} is verified. 
\end{proof}
Since the above-mentioned result of Wang plays a critical role in the proof of theorem \ref{nonparabolic}, we reformulate it here for reader's convenience.
\begin{theorem}[\cite{W1}]
Let $(M^n, g, e^{-f}dv)$ be a smooth metric measure space of dimension $n\geq3$. Suppose that the lower bound of the spectrum $\lambda_1(M)$ of the weighted Laplacian is positive and 
$$Ric_f^m\geq-\frac{m-1}{m-2}\lambda_1(M).$$
Then either $M$ has only one $p$-nonparabolic end or $M=\RR\times N^{n-1}$ for some compact manifold $N$ of dimension $n-1$ with the product metric
$$ ds_M^2=dt^2+\cosh\ \sqrt[]{\frac{\lambda_1(M)}{m-2}}tds_N^2. $$
\end{theorem}
\section{Bakry-\'{E}mery curvatures and maximality of the weighted $p$-spectrum}
\setcounter{equation}{0}
In this section, we will show the rigidity of smooth metric measure spaces with Bakry-\'{E}mery curvature $Ric_f$. We will estimate a upper bound of $\lambda_{1, p}$ and show that this upper bound is optimal  and $\lambda_{p,f}$ is maximal on gradient staedy Ricci soliton with radial Ricci flatness. Moreover, we will prove that under extra conditions on $Ric_f$ and $f$ if $\lambda_{1, p}$ is maximal then smooth metric measure spaces are splliting. 

\begin{lemma}\label{upper1}
Let $(M, g, e^{-f}dv)$ be a complete noncompact smooth metric measure space with $Ric_f\geq0$. If there exist positive constants $a, b>0$ such that
$$ f(x)\leq ar(x)+b\quad\text{for all }x\in M $$
then we have the upper bound estimate
$$ \lambda_{1,p}(M)\leq\frac{a^{p}}{p^{p}}. $$
Moreove, if $f$ has sublinear growth then $\lambda_{1,p}(M)=0$.
\end{lemma}
\begin{proof}We use the argument in the proof of Lemma \ref{eigenvalue}. Similarly, without loss of generality, we may assume that $\lambda_{1,p}$ is positive. By the variational characterization of $\lambda_{1,p}$, we know that $M$ has infinite $f$-volume. The theorem 0.1 in \cite{BK} implies $M$ is $p$-nonparabolic, moreover 
$$ V_f(B_p(R))\geq Ce^{p\lambda_{1,p}^{1/p}R}, $$
for all sufficiently large $R$ and $C$ is a constant dependent on $r$. On the other hand, by the volume estimate in \cite{MW1}, we have known that
$$ V_f(B_p(R))\leq C'R^{n}e^{aR}, $$
for some constant $C'>0$. Therefore, we have for some constant $C>0$
$$ e^{p\lambda_{1,p}^{1/p}R}\leq CR^ne^{aR}, $$
or equivalently,
$$ \lambda_{1,p}^{1/p}\leq\frac{\ln C+n\ln R}{pR}+\frac{a}{p} $$
for all sufficiently large $R$. Lettting $R\to \infty$, we have
$$ \lambda_{1,p}\leq\left(\frac{a}{p}\right)^p. $$
The proof is complete. 
\end{proof}
To show that the estimate $\lambda_{1,p}(M)$ as in the above is optimal, let us recall the following Picone identity (\cite{alleg}). 
\begin{lemma}[\cite{alleg}]\label{picone}
Let $M$ be a Riemannian manifold and $\Omega\subset M$ is a open subset. Suppose that $v>0, u\geq0$ be differentiable in $\Omega$. Denote
$$ \begin{aligned}
L(u, v)&=|\nabla u|^{p}+(p-1)\frac{u^{p}}{v^{p}}|\nabla v|^{p}-p\frac{u^{p-1}}{v^{p-1}}\left\langle\nabla u, |\nabla v|^{p-2}\nabla v\right\rangle\\
R(u, v)&=|\nabla u|^{p}-\left\langle\nabla\left(\frac{u^{p}}{v^{p-1}}\right), |\nabla v|^{p-2}\nabla v\right\rangle.  
\end{aligned} $$
Then 
$$ L(u, v)= R(u, v). $$
Moreover, $L(u, v)\geq0$, and $L(u, v)=0$ a.e. $\Omega$ if and only if $\nabla(u/v)=0$ a.e. $\Omega$, i.e.
$u=kv$ for some constant $k$ in each component of $\Omega$.
\end{lemma}
Now, we can give a lower bound of $\lambda_{1, p}$ as follows.
\begin{lemma}\label{lowerbound}
Let $(M, g, e^{-f}dv)$ be a smooth metric measure space. If there exists
a positive function $v>0$ such that $\Delta_{p, f}v\leq-\lambda v^{p-1}$ for some constant $\lambda>0$, then
$$ \lambda_{1,p}(M)\geq\lambda. $$
\end{lemma}
\begin{proof}
Let $\left\{\Omega_j\right\}_{j=1}^{\infty}$ be a exhaustion of $M$ by compact domains $\Omega_j$. Choose $\varphi\in{\mathcal C}^{\infty}_0(\Omega_j), \varphi\geq0$. Then, we have
$$ \begin{aligned}
0\leq \int_{\Omega_j}L(\varphi, v)e^{-f}dv
&=\int_{\Omega_j}R(\varphi, v)=\int_{\Omega_j}|\nabla \varphi|^{p}e^{-f}dv-\int_{\Omega_j}\left\langle\nabla\left(\frac{\varphi^{p}}{v^{p-1}}\right), |\nabla v|^{p-2}\nabla v\right\rangle e^{-f}dv\\
&=\int_{\Omega_j}|\nabla \varphi|^{p}e^{-f}dv+\int_{\Omega_j}\frac{\varphi^{p}}{v^{p-1}}\Delta_{p,f}ve^{-f}dv\\
&\leq\int_{\Omega_j}|\nabla \varphi|^{p}e^{-f}dv-\lambda\int_{\Omega_j}\varphi^{p}e^{-f}dv.
\end{aligned} $$
This shows that $\lambda\leq\lambda_{1,p}(\Omega_j)$. Since $\lambda_{1,p}(M)=\lim\limits_{j\to\infty}\lambda_{1, p}(\Omega_j)$. The proof is complete.
\end{proof}
Recall that a gradient steady Ricci soliton has the below properties
\begin{equation}\label{gsrs}
\begin{cases}
|\nabla f|^{2}+S&=a^{2},\quad\text{ for some constant }a>0\\
\Delta f+S&=0\\
S&\geq0
\end{cases}
\end{equation}
where $S$ is the scalar curvature of $M$. Moreover, $Ric+Hess f=0$. As in \cite{pw}, a gradient Ricci soliton is called to have radial Ricci flatness if 
$$ Ric(\nabla f, \nabla f)=0. $$
In \cite{pw}, Petersen and Wylie pointed out that a compact gradient steady Ricci soliton has rigidity property provided the radial Ricci flatness. The following theorem tells us the sharpness of the upper bound of $\lambda_{1,p}(M)$ on noncompact gradient steady Ricci flat.  
\begin{theorem}\label{sharp1}
Let $(M, g, f)$ be a noncompact gradient steady Ricci soliton with radial Ricci flatness normalized as \eqref{gsrs}. Then 
$$ \lambda_{1,p}=\frac{a^{p}}{p^{p}} $$
provided that one of the following conditions holds true 
\begin{description}
\item (i) $1<p<2$
\item (ii) $p\geq 2$ and $|\nabla f|\geq a\ \sqrt[]{\frac{p-2}{p-1}}$.
\end{description}
\end{theorem}
\begin{proof}
Since $S\geq0$, we infer $|\nabla f|\leq a$. Hence, Lemma \ref{upper1} implies
$$ \lambda_{1, p}(M)\leq\left(\frac{a}{p}\right)^{p}. $$
We will show that the inverse inequality is also  true. Indeed, let $v=e^{\frac{1}{p}f}$, it is easy to compute that
$$ \begin{aligned}
\Delta_{p,f}v=\frac{1}{p^{p-1}}\left(|\nabla f|^{p-2}\Delta_f+\frac{p-1}{p}|\nabla f|^{p}\right)v^{p-1}+\frac{v^{p-1}}{p^{p-1}}\left\langle\nabla|\nabla f|^{p-2}, \nabla f\right\rangle. 
\end{aligned} $$
Since $M$ is a gradient steady Ricci soliton,
$$ \Delta_ff=\Delta_f-|\nabla f|^{2}=-(S+|\nabla f|^{2})=-a^{2}. $$
This follows that
$$\begin{aligned}
 \Delta_{p,f}v
&=\frac{1}{p^{p-1}}\left(-a^{2}|\nabla f|^{p-2}+\frac{p-1}{p}|\nabla f|^{p}\right)v^{p-1}+(p-2)\frac{v^{p-1}}{p^{p-1}}|\nabla f|^{p-3}Hess(\nabla f, \nabla f)\\
&= \frac{1}{p^{p-1}}\left(-a^{2}|\nabla f|^{p-2}+\frac{p-1}{p}|\nabla f|^{p}\right)v^{p-1}. 
\end{aligned}$$
Here we used $M$ is a gradient Ricci soliton and radial Ricci flatness of $M$ in the last equality.

Observe that if either $1< p <2$ and $0<x\leq a$ or; $p\geq 2$ and $a\ \sqrt[]{\frac{p-2}{p-1}}\leq x\leq a$, then we have 
$$ -a^{2}x^{p-2}+\frac{p-1}{p}x^{p}\leq-\frac{a^{p}}{p}. $$
Hence, we conclude that $\Delta_{p, f}v\leq -\frac{a^{p}}{p^{p}}v^{p-1}$.  By Lemma \eqref{lowerbound}, this implies $\lambda_{1, p}\geq\left(\frac{a}{p}\right)^{p}$. The proof is complete.
\end{proof}
\begin{lemma}\label{weightedend}
Let $(M, g, e^{-f}dv)$ be a smooth metric measure space with $Ric_f\geq0$, Then $M$ has at most one weighted $p$-nonparabolic end.
\end{lemma}
\begin{proof}Suppose $M$ has two f-nonparabolic ends. Then $M$ admits
a positive non-constant bounded weighted $p$-harmonic function $v$ with
$\int_M|\nabla v|^{p}e^{-f}dv<\infty$. In fact, this kind of result was first discovered by Li and Tam in \cite{litam} (see also \cite{CCW, MW1}). The Bochner formular for the form $|dv|^{p-2}dv$ implies
\begin{align}
\frac{1}{2}\Delta_f(||dv|^{p-2}dv|^{2})
&=|\nabla(|dv|^{p-2}dv)|^{2}-\left\langle(\delta_f d+d\delta_f)(|dv|^{p-2}dv), |dv|^{p-2}dv\right\rangle\notag\\
&\quad+Ric_f(|dv|^{p-2}dv, |dv|^{p-2}dv). \notag\\
&\geq |\nabla(|dv|^{p-2}dv)|^{2}-\left\langle \delta_f d(|dv|^{p-2}dv), |dv|^{p-2}dv\right\rangle\notag\\
&\geq |\nabla(|dv|^{p-1})|^{2}-\left\langle \delta_f d(|dv|^{p-2}dv), |dv|^{p-2}dv\right\rangle\notag\\
&=(p-1)^{2}|dv|^{2(p-2)}|\nabla|dv||^{2}-\left\langle \delta_f d(|dv|^{p-2}dv), |dv|^{p-2}dv\right\rangle
\label{star1}
\end{align}
Here $\delta_f=\delta +i_{\nabla f}$ is the dual of $d$ with respect to $e^{-f}dv$, we denoted $1$-form and its dual vector field by $dv$ in an abuse of notation. We also used $v$ is weighted $p$-harmonic in the first inequality.

Let $M_+:=\left\{x\in M: |dv|(x)\not=0\right\}$ and $\varphi\in{\mathcal C}_0^{\infty}(M)$ such that $0\leq\varphi\leq1, \varphi=1$ on $B(p, 2R)$, $|\nabla\phi|\leq\frac{2}{R}$ on $B(p, 2R)$. For any $\eps>0$, we denife
$$ \psi_\eps=\begin{cases}
\dfrac{|du|(x)}{\max\left\{|du|(x), \eps\right\}}&\quad\text{ on }M_+\\
0&\quad\text{ on }M\setminus M_+.
\end{cases} $$
As in \cite{nak}, $\psi_\eps\in L_0^{1,2}(M_+)$, where $M_+$ is the completion of ${\cal C}_0^{\infty}(M_+)$. Moreover, $\psi_\eps\to 1, \nabla\psi_\eps\to 0$ in $L_0^{1, 2}(M_+)$. Multiplying both sides of \eqref{star1} by $|dv|^{2-p}\psi_\eps\varphi^2e^{-f}$ then integrate the obtained result over $M_+$, we have
\begin{align}
2(p-1)^{2}&\int_{M_+}|dv|^{p-2}|\nabla|dv||^{2}\psi_\eps\varphi^{2}e^{-f}\notag\\
&\leq\int_{M_+}|dv|^{2-p}\psi_\eps\varphi^{2}\Delta_f(|dv|^{2(p-1)})e^{-f}
 +\int_{M_+}\left\langle \delta_fd(|dv|^{p-2}dv), \psi_\eps\varphi^{2}dv\right\rangle e^{-f}\notag \\
&=-\int_{M_+}\left\langle\nabla(|dv|^{2-p}\psi_\eps\varphi^{2}), \nabla(|dv|^{2(p-1)})\right\rangle e^{-f}
+\int_{M_+}\left\langle d(|dv|^{p-2}dv), d(\psi_\eps\varphi^{2}dv)\right\rangle e^{-f} \label{star2}
\end{align}
It is proved in \cite{Rimoldi} that $|dA\wedge dA|\leq|dA||dB|$ for any $1$-forms $dA, dB$. Hence, we can estimate the second term of the right hand side of \eqref{star2} as follows.
$$ \begin{aligned}
\int_{M_+}\left\langle d(|dv|^{p-2}dv), d(\psi_\eps\varphi^{2}dv)\right\rangle e^{-f}
&=\int_{M_+}\left\langle d(|dv|^{p-2})\wedge dv, d(\psi_\eps\varphi^{2})\wedge dv\right\rangle e^{-f}\\
&\leq\int_{M_+}|d(|dv|^{p-2})\wedge dv|.|d(\psi_\eps\varphi^{2})\wedge dv|e^{-f}\\
&\leq\int_{M_+}|d(|dv|^{p-2})||dv|^{2}|\nabla(\psi_\eps\varphi^{2})|e^{-f}\\
&=|p-2|\int_{M_+}|dv|^{p-1}||\nabla|dv||\nabla(\psi_\eps\varphi^{2})|e^{-f}.
\end{aligned} $$
The first term of the right hand side of \eqref{star2} can be estimated by
$$ \begin{aligned}
-\int_{M_+}&\left\langle\nabla(|dv|^{2-p}\psi_\eps\varphi^{2}), \nabla(|dv|^{2(p-1)})\right\rangle e^{-f}\\
&=-\int_{M_+}\left\langle \nabla(|dv|^{2-p}), \nabla(|dv|^{2(p-1)})\right\rangle \psi_\eps\varphi^{2}e^{-f}-\int_{M_+}|dv|^{2-p}\left\langle\nabla(\psi_\eps\varphi), \nabla(|dv|^{2(p-1)})\right\rangle e^{-f}  \\
&\leq 2(p-1)(p-2)\int_{M_+}|dv|^{p-2}|\nabla|dv||^{2}\psi_\eps\varphi^{2}e^{-f}+2(p-1)\int_{M_+}|dv|^{p-1}|\nabla|dv|||\nabla(\psi_\eps\varphi^{2})|e^{-f}.
\end{aligned} $$
Combining the above two inequalities and \eqref{star2}, we conclude that there is a constant $C=C(p)$ depending only on $p$, such that
$$\begin{aligned}
 2(p-1)&\int_{M_+}\varphi^{2}|dv|^{p-2}|\nabla|dv||^{2}\psi_\eps\varphi^{2}e^{-f}\\
&\leq C\int_{M_+}|dv|^{p-1}|\nabla|dv|||\nabla(\psi_\eps\varphi^{2})|e^{-f} \\
&\leq C\int_{M_+}|dv|^{p-1}|\nabla|dv|||\nabla\psi_\eps|\varphi^{2}e^{-f}+2C\int_{M_+}|dv|^{p-1}|\nabla|dv|||\nabla\varphi|\psi_\eps e^{-f}.
\end{aligned}$$
Let $\eps\to0$ we have that 
$$ 2(p-1)\int_{M_+}\varphi^{2}|dv|^{p-2}|\nabla|dv||^{2}e^{-f}\leq 2C\int_{M_+}|dv|^{p-1}|\nabla|dv|||\nabla\varphi|e^{-f} $$
for some positive constant $C$. Using the following fundamental inequality
$$ 2|dv|^{p-1}|\nabla|dv|||\nabla\varphi|\leq \frac{p-1}{C}|dv|^{p-2}|\nabla|dv||^{2}+\frac{C}{p-1}|dv|^{p}|\nabla\varphi|^{2}, $$
we infer 
$$ (p-1)\int_{M_+}\varphi^{2}|dv|^{p-2}|\nabla|dv||^{2}e^{-f}\leq\frac{C^{2}}{p-1}\int_{M_+}|dv|^{p}|\nabla\varphi|^{2}e^{-f}. $$
Hence
$$ \int_{M_+}\varphi^{2}|dv|^{p-2}|\nabla|dv||^{2}e^{-f}\leq\frac{C^{2}}{(p-1)^{2}R^2}\int_{(M_+)\cap B(p, 2P)}|dv|^{p}e^{-f}\leq \frac{C^{2}}{(p-1)^{2}R^2}\int_{M}|dv|^{p}e^{-f}.$$
Let $R\to\infty$, we conclude that $|dv|$ is constant on $M_+$. Since $v\in{\mathcal C}^{1}(M)$, it implies that $dv=0$ on $M$. Hence $v$ is constant. This gives a contradiction. The proof is complete. 
\end{proof}
Now, we have the following rigidity result.
\begin{theorem}\label{rigidity1}
Let $(M, g, e^{-f}dv)$ be a smooth metric measure space with $Ric_f\geq0$. Suppose that $\lambda_{1,p}=\frac{a^{p}}{p^{p}}$, where $a$ is the linear growth rate of $f$. Then, either $M$ is connected at infinity or $M=\RR\times N$ where $N$ is a compact manifold.
\end{theorem}
\begin{proof}Since $\lambda_{1, p}>0$, we know that $M$ is weighted $p$-nonparabolic. Assume that $M$ has at leat two ends. By Lemma \ref{weightedend}, we infer that $M$ has only one weighted $p$-nonparabolic end, all other are weighted $p$-parabolic. Let $E$ be the weighted $p$-nonparabolic end and $F=M\setminus E$. 

Let $\beta$ be the Busemann function associated with a geodesic ray $\gamma$ contained in $F$, namely,
$$ \beta(x)=\lim\limits_{t\to\infty}(t-dist(x, \gamma(t))). $$
In \cite{MW1}, Munteanu and Wang proved that 
$$ \Delta_f\beta\geq-a. $$
Hence,
$$ \begin{aligned}
\Delta_{p,f}\left(e^{\frac{a}{p}\beta}\right)
&=e^fdiv\left(e^{-f}\left(\frac{a}{p}\right)^{p-2}e^{\frac{a}{p}(p-2)\beta}\nabla e^{\frac{a}{p}\beta}\right)\\
&=\left(\frac{a}{p}\right)^{p-1}e^{\frac{a}{p}(p-1)\beta}\Delta_f\beta+\left(\frac{a}{p}\right)^{p-1}\left(\frac{a}{p}\right)(p-1)e^{\frac{a}{p}(p-1)\beta}|\nabla\beta|^2\\
&\geq a\left(\frac{a}{p}\right)^{p-1}e^{\frac{a}{p}(p-1)\beta}\left(\frac{a}{p}-1\right)\\
&=-\left(\frac{a}{p}\right)^pe^{\frac{a}{p}(p-1)\beta}.
\end{aligned} $$
Therefore, let $\omega:=e^{\frac{a}{p}\beta}$, we obtain 
$$ \Delta_{p,f}(\omega)\geq-\lambda_{1,p}\omega^{p-1}. $$
Suppose that $\phi$ is a nonnegative compactly supported smooth function on $M$. Then by the variational principle,
$$ \lambda_{1,p}\int_M(\phi\omega)^pe^{-f}\leq\int_M|\nabla(\phi\omega)|^pe^{-f}. $$
Noting that, integration by parts implies
$$ \int_M\phi^p\omega\Delta_{p,f}(\omega) e^{-f}=-\int_M\phi^p|\nabla\omega|^{p}e^{-f}-p\int_M\phi^{p-1}\omega\left\langle\nabla\phi,\nabla\omega\right\rangle|\nabla\omega|^{p-2}e^{-f}  $$
and
$$ \begin{aligned}
|\nabla(\phi\omega)|^p
&=\left(|\nabla\phi|^2\omega^2+2\phi\omega\left\langle\nabla\phi, \nabla\omega\right\rangle+\phi^2|\nabla\omega|^{p-2} \right)^\frac{p}{2}\\
&\leq \phi^p|\nabla\omega|^p+p\phi\omega\left\langle\nabla\phi, \nabla\omega\right\rangle\phi^{p-2}|\nabla\omega|^{p-2}+c|\nabla\phi|^2\omega^p 
\end{aligned} $$
for some constant $c$ depending only on $p$, we infer
\begin{align}
&\int_M\phi^p\omega(\Delta_{p,f}(\omega)+\lambda_{1,p}\omega^{p-1})e^{-f}\notag\\
&=\lambda_{1,p}\int_M(\phi\omega)^pe^{-f}-\int_M\phi^p|\nabla\omega|^pe^{-f}-p\int_M\phi^{p-1}\omega\left\langle\nabla\phi,\nabla\omega\right\rangle|\nabla\omega|^{p-2}e^{-f}\notag\\
&\leq\int_M|\nabla(\phi\omega)|^pe^{-f}-\int_M\phi^p|\nabla\omega|^pe^{-f} -p\int_M\phi^{p-1}\omega\left\langle\nabla\phi,\nabla\omega\right\rangle|\nabla\omega|^{p-2}e^{-f}\notag\\
&\leq c\int_M|\nabla\phi|^2\omega^pe^{-f}.\label{swe311}
\end{align} 
Now, we choose
$$ \phi=\begin{cases}
1&\quad\text{in }B(R)\\
0,&\quad\text{ on }M\setminus B(2R)
\end{cases} $$
such that $|\nabla\phi|\leq\frac{2}{R}$. Then we conclude
\begin{align}
\int_M|\nabla\phi|^2\omega^pe^{-f}
&=\int_M|\nabla\phi|^2e^{a\beta}e^{-f}\notag\\
&\leq\frac{4}{R^2}\int_{B(2R)\setminus B(R)}e^{a\beta}e^{-f}\notag\\
&=\frac{4}{R^2}\int_{F\cap(B(2R)\setminus B(R))}e^{a\beta}e^{-f}+\frac{4}{R^2}\int_{(M\setminus F)\cap(B(2R)\setminus B(R))}e^{a\beta}e^{-f}.\label{swe321}
\end{align}
Since $\lambda_{1,p}=\left(\frac{a}{p}\right)^p$, by the proof of Theorem 4.1 in \cite{MW1} (or by theorem 0.1 in \cite{BK}), it turns out that
$$ V_f(F\setminus B(R))\leq  Ce^{-aR}. $$
Hence, the first term of \eqref{swe321} tends to $0$ as $R$ goes to $\infty$. Moreover, by the proof of Theorem 4.1 in \cite{MW1}, we also know that
$$ V_f(B(R)\cap E)\leq ce^{aR}. $$
Therefore, the second term of \eqref{swe321} also goes to $0$ as $R\to\infty$. Consequently, the inequality \eqref{swe311} implies 
$$ \Delta_{p,f}(\omega)+\lambda_{1,p}\omega^{p-1}\equiv0. $$
This implies 
$$ \Delta_f\beta=-a \text{ and }|\nabla\beta|=1. $$
holds every where on $M$. Using the argument in \cite{MW1}, we conclude that $M=\RR\times N^{n-1}$ for some compact manifold $N$ of dimension $n$. The proof is complete.
\end{proof}
\begin{theorem}\label{lower2}
Let $(M, g, e^{-f}dv)$ be a complete smooth metric measure space with $Ric_f\geq-(n-1)$. Suppose that the linear growth rate of $a$ is $a$. Then we have
$$ \lambda_{1,p}(M)\leq\left(\frac{n-1+a}{p}\right)^{p} $$
In particular, if $f$ is of sublinear growth, then the bottom spectrum of the weighted $p$-Laplacian has the following sharp upper bound:
$$ \lambda_{1,p}(M)\leq\left(\frac{n-1}{p}\right)^{p}. $$
\end{theorem}
The sharpness of the $\lambda_{1,p}$ is desmontrated by the following example.
\begin{example}
Let $M=\RR\times N$with a warped product metric
$$ ds^{2}=dt^{2}+e^{2t}ds_N^{2}, $$
where $N$ is a complete manifold with non-negative Ricci curvature. Then it can be directly checked that $Ric_M\geq-(n-1)$. Choose wighted function $f=-at$, then the Bakry-\'{E}mery curvature is bounded by $Ric_f\geq-(n-1)$. Let $v(t,x)=e^{-bt}$, where $b=\frac{n-1+a}{p}$ then 
$$ \begin{aligned}
e^fdiv(e^{-f}|\nabla v|^{p-2}\nabla v)
&=|\nabla v|^{p-2}\Delta v+\left\langle \nabla|\nabla v|^{p-2}, \nabla v\right\rangle-|\nabla v|^{p-2}\left\langle\nabla v, \nabla f\right\rangle\\
&=(1-n+b)b^{p-1}v^{p-1}+(p-2)b^{p}v^{p-1}-ab^{p-1}v^{p-1}\\
&=(b(p-1)-(n-1+a))b^{p-1}v^{p-1}\\
&=-\left(\frac{n-1+a}{p}\right)^{p}v^{p-1}
\end{aligned} $$
Hence, by Lemma \ref{lowerbound}, we conclude $\lambda_{1,p}\geq\left(\frac{n-1+a}{p}\right)^{p}$.
\end{example}
Now, we will a proof of the theorem \ref{lower2}. 
\begin{proof}
We follow the proof of Lemma \ref{upper1}. Without loss of generality, we may assume that $\lambda_{1,p}$ is positive. By the variational characterization of $\lambda_{1,p}$, we know that $M$ has infinite $f$-volume. The theorem 0.1 in \cite{BK} implies $M$ is $p$-nonparabolic, moreover 
$$ V_f(B_p(R))\geq Ce^{p\lambda_{1,p}^{1/p}R}, $$
for all sufficiently large $R$ and $C$ is a constant dependent on $r$. On the other hand, by the volume estimate in \cite{MW2}, we have known that
$$ V_f(B_p(R))\leq C'e^{(n-1+a)R}, $$
for some constant $C'>0$. Therefore, we have for some constant $C>0$
$$ e^{p\lambda_{1,p}^{1/p}R}\leq Ce^{(n-1+a)R}, $$
for all $R$. This implies
$$ \lambda_{1,p}\leq\left(\frac{n-1+a}{p}\right)^p. $$
The proof is complete.
\end{proof}
\begin{theorem}\label{riccif}
Let $(M^{n}, g, e^{-f}dv)$ be a smooth metric measure space with $Ric_f\geq0$ and $n\geq3$. Assume that $|\nabla f|\leq a$ on $M$ for some constant $a\geq0$. If $\lambda_{1,p}(M)=\left(\frac{n-1+a}{p}\right)^{p}$ then either $M$ has no weighted $p$-parabolic end; or $M$ is a warped product $M=\RR\times N$, where $N$ is compact.  
\end{theorem}
\begin{proof}We follow the proof of Theorem \ref{theorem1}. Suppose that $M$ has a weighted $p$-parabolic end $E$. Let $\beta$ be the Busemann function associated with a geodesic ray $\gamma$ contained in $E$, namely,
$$ \beta(x)=\lim\limits_{t\to\infty}(t-dist(x, \gamma(t))). $$
By the proof of Theorem 4.1 in \cite{MW2}, we have 
$$ \Delta_f\beta\geq-(n-1+a) $$
in the sense of distributions. Hence,
$$ \begin{aligned}
\Delta_{p,f}\left(e^{\frac{n-1+a}{p}\beta}\right)
&=e^fdiv\left(e^{-f}\left(\frac{n-1+a}{p}\right)^{p-2}e^{\frac{n-1+a}{p}(p-2)\beta}\nabla e^{\frac{n-1+a}{p}\beta}\right)\\
&=\left(\frac{n-1+a}{p}\right)^{p-1}e^{\frac{n-1+a}{p}(p-1)\beta}\Delta_f\beta+\left(\frac{n-1+a}{p}\right)^{p-1}\left(\frac{n-1+a}{p}\right)(p-1)e^{\frac{n-1+a}{p}(p-1)\beta}|\nabla\beta|^2\\
&\geq(n-1+a)\left(\frac{n-1+a}{p}\right)^{p-1}e^{\frac{n-1+a}{p}(p-1)\beta}\left(\frac{p-1}{p}-1\right)\\
&=-\left(\frac{n-1+a}{p}\right)^pe^{\frac{n-1+a}{p}(p-1)\beta}.
\end{aligned} $$
Hence, let $\omega:=e^{\frac{n-1+a}{p}\beta}$, we obtain 
$$ \Delta_{p,f}(\omega)\geq-\lambda_{1,p}\omega^{p-1}. $$
Suppose that $\phi$ is a nonnegative compactly supported smooth function on $M$. Then by the proof of Theorem \ref{theorem1}, we have
\begin{align}
&\int_M\phi^p\omega(\Delta_{p,f}(\omega)+\lambda_{1,p}\omega^{p-1})e^{-f}\notag\\
&=\lambda_{1,p}\int_M(\phi\omega)^pe^{-f}-\int_M\phi^p|\nabla\omega|^pe^{-f}-p\int_M\phi^{p-1}\omega\left\langle\nabla\phi,\nabla\omega\right\rangle|\nabla\omega|^{p-2}e^{-f}\notag\\
&\leq\int_M|\nabla(\phi\omega)|^pe^{-f}-\int_M\phi^p|\nabla\omega|^pe^{-f} -p\int_M\phi^{p-1}\omega\left\langle\nabla\phi,\nabla\omega\right\rangle|\nabla\omega|^{p-2}e^{-f}\notag\\
&\leq c\int_M|\nabla\phi|^2\omega^pe^{-f}.\label{swe3111}
\end{align} 
Now, we choose
$$ \phi=\begin{cases}
1&\quad\text{in }B(R)\\
0,&\quad\text{ on }M\setminus B(2R)
\end{cases} $$
such that $|\nabla\phi|\leq\frac{2}{R}$. Then we conclude
\begin{align}
\int_M|\nabla\phi|^2\omega^pe^{-f}
&=\int_M|\nabla\phi|^2e^{(m-1)\beta}e^{-f}\notag\\
&\leq\frac{4}{R^2}\int_{B(2R)\setminus B(R)}e^{(m-1)\beta}e^{-f}\notag\\
&=\frac{4}{R^2}\int_{E\cap(B(2R)\setminus B(R))}e^{(m-1)\beta}e^{-f}+\frac{4}{R^2}\int_{(M\setminus E)\cap(B(2R)\setminus B(R))}e^{(m-1)\beta}e^{-f}.\label{swe3211}
\end{align}
Since $\lambda_{1,p}=\left(\frac{n-1+a}{p}\right)^p$, by the proof of Theorem 4.1 in \cite{MW2}, it turns out that
$$ V_f(E\setminus B(R))\leq  \widehat{c}e^{-(n-1+a)R}. $$
Hence, the first term of \eqref{swe3211} tends to $0$ as $R$ goes to $\infty$. On the other hand, by \cite{LW3} we have 
$$ \beta(x)\leq-r(x)+\widetilde{c} $$
on $M\setminus E$. The Laplacian comparison theorem in \cite{MW2} implies $V_f(E\cap B(R))\leq c'e^{(n-1+a)R}$. It turns out that the second term of \eqref{swe3211} also goes to $0$ as $R\to\infty$. Therefore, \eqref{swe3111} infers
$$ \Delta_{p,f}(\omega)+\lambda_{1,p}\omega^{p-1}\equiv0. $$
This implies 
$$ \Delta_f\beta=-(n-1+a) \text{and }|\nabla\beta|=1. $$
By the argument in \cite{MW2}, we conclude that $M=\RR\times N^{n-1}$ for some compact manifold $N$ of dimension $n$. The proof is complete.
\end{proof}
\vskip0.5cm
\noindent
\subsection*{Acknowledgment}
A part of this paper was done during a visit of the first author to Vietnam Institute for Advanced Study in Mathematics (VIASM) and Institut Fourier, Grenoble, France. He would like to express his thanks to staffs there for the excellent working conditions, and the financial support. He also would like to express his deep gratitude to his advisor Prof. Chiung Jue Sung for her constant support. The authors also thank Munteanu and J. Y. Wu for their usefull comments on the earlier manuscript of this paper. This work is supported in part by NAFOSTED under grant number 101.02-2014.49. 
\vskip0.5cm

\addcontentsline{toc}{section}{3 \hspace{0.09cm} References}

\vskip0.4cm

\bigskip
\noindent
Nguyen Thac Dung,\\
 \textit{\ 
Department of Mathematics - Mechanics - Informatics (MIM),
Hanoi University of Sciences (HUS-VNU),
No. 334, Nguyen Trai Road,
Thanh Xuan, Hanoi, Vietnam.}\\
{\small E-mail address: dungmath@yahoo.co.uk}

\vskip0.4cm

\bigskip
\noindent
Nguyen Duy Dat,\\
 \textit{\ 
Department of Mathematics - Mechanics - Informatics (MIM),
Hanoi University of Sciences (HUS-VNU),
No. 334, Nguyen Trai Road,
Thanh Xuan, Hanoi, Vietnam.}\\
{\small E-mail address: duynguyendat@gmail.com}

\end{document}